\documentclass[12pt,a4paper, reqno]{article}

\usepackage[top=1in, bottom=1in, left=1in, right=1in]{geometry}

\usepackage{indentfirst} 
\usepackage{amsmath, amsthm, amssymb,xfrac, mathrsfs} 
\usepackage{tikz} 
\usetikzlibrary{hobby}
\usepackage{wrapfig} 
\usepackage{xcolor} 
\usepackage{verbatim} 
\usepackage{enumerate} 

\usepackage{hyperref} 
\hypersetup{
colorlinks=true,
citecolor=black!50!red,
linkcolor=black!50!red,
linktoc=all
}
\usepackage[all]{hypcap} 

\newcounter{constant}
\newcommand{\newconstant}[1]{\refstepcounter{constant}\label{#1}}
\newcommand{\useconstant}[1]{c_{\textnormal{\tiny \ref{#1}}}}
\newcounter{bigconstant}
\newcommand{\newbigconstant}[1]{\refstepcounter{bigconstant}\label{#1}}
\newcommand{\usebigconstant}[1]{C_{\textnormal{\tiny \ref{#1}}}}
\newtheorem{teo}{Theorem}[section]
\newtheorem{prop}[teo]{Proposition}
\newtheorem{lemma}[teo]{Lemma}
\newtheorem{claim}[teo]{Claim}
\newtheorem{cor}[teo]{Corollary}

\numberwithin{equation}{section} 

\theoremstyle{definition}

\newtheorem{remark}[teo]{Remark}

\newcommand{\PP}{\mathbb{P}}
\newcommand{\EE}{\mathbb{E}}
\newcommand{\RR}{\mathbb{R}}

\newcommand{\NN}{\mathbb{N}}
\newcommand{\ZZ}{\mathbb{Z}}

\newcommand{\charf}[1]{\mathbf{1}_{#1}}

\DeclareMathOperator{\dist}{d}

\DeclareMathOperator{\poisson}{Poisson}

\begin{document}

\title{Spread of an infection on the zero range process}

\author{Rangel Baldasso\footnote{Email: \ baldasso@impa.br; \ Bar-Ilan University, 5290002, Ramat Gan, Israel} \and Augusto Teixeira\footnote{Email: \ augusto@impa.br; \ IMPA, Estrada Dona Castorina 110, 22460-320,\newline Rio de Janeiro, RJ, Brazil}}

\maketitle

\begin{abstract}
We study the spread of an infection on top of a moving population. The environment evolves as a zero range process on the integer lattice starting in equilibrium. At time zero, the set of infected particles is composed by those which are on the negative axis, while particles at the right of the origin are considered healthy. A healthy particle immediately becomes infected if it shares a site with an infected particle. We prove that the front of the infection wave travels to the right with positive and finite velocity. As a central step in the proof of these results, we prove a space-time 
decoupling for the zero range process which is interesting on its own. Using a sprinkling technique, we derive an estimate on the correlation of functions of the space of trajectories whose supports are sufficiently far away.
\end{abstract}

\bigskip

\noindent \textup{2010} Mathematics Subject Classification: 60K37, 60K35, 82C22.


\bigskip

\section{Introduction}\label{zrpsec:intro}

\par There are many mathematical models for the spread of an infectious disease. They are interesting not only because of their potential applications, but also as mathematical objects themselves, since the understanding of such models usually requires some new and exciting mathematics. There are deterministic ways of modeling such diseases, such as the SIR model, see~\cite{sir}, as well as stochastic modeling, such as the contact process,~\cite{liggett} and~\cite{bg}, and the $x+y \to 2x$ model, considered in~\cite{rs},~\cite{cqr} and~\cite{br1}. We consider here an infection model that evolves on top of a zero range process in the one-dimensional integer lattice.

\par The zero range process in $\ZZ$ is a system where particles interact only when they are at the same site. The interaction controls the rate with which particles leave the site and this rate is given by a function  $g:\NN_{0} \to \RR_{+}$ of the number of particles with $g(0)=0$. Particles jump to a uniformly chosen nearest neighbor. We defer the precise definition of the model to Section~\ref{sec:zrp}.

\par Assume that there exist positive constants $\Gamma_{-} \leq 1 \leq \Gamma_{+}$ such that
\begin{equation}\label{eq:g_condition}
\Gamma_{-} \leq g(k)-g(k-1) \leq \Gamma_{+}, \qquad \text{for all } k \in \NN.
\end{equation}
In this case, there exist explicit formulas for the invariant measures of the zero range process, see Chapter 2 of~\cite{kipnis_landim}. In fact, Assumption~\eqref{eq:g_condition} implies that, for every $\rho \in \RR_{+}$, there exists an associated product invariant measure with density $\rho$. Let $\EE_{\rho}$ denote the expectation with respect to the distribution of a zero range process with initial state given by the invariant measure with density $\rho$.

Given an initial configuration $\eta_{0}$ for the zero range process, we declare all particles to the left of zero, $\xi_{0}$, infected. Define also $\zeta_{0}=\eta_{0}-\xi_{0}$ as the configuration of healthy particles.

\par We assume that the process $\xi+\zeta$ evolves as a zero range process with rate function $g$. Besides, a healthy particle becomes immediately infected when it shares a site with some already infected particle. In particular, in any non-empty site, either all particles are healthy or all particles are infected. Denote by $\xi_{t}$ the configuration of infected particles at time $t$.

\par We define the front of the infection wave as
\begin{equation}
r_{t}=\sup\{x:\xi_{t}(x)>0\}.
\end{equation}
If $\rho>0$, and we choose $\eta_{0}$ according to the invariant measure with density $\rho$, then $r_{t} \in \ZZ$ for all $t \geq 0$ almost surely.

\par The existence of dependencies in the jump rates makes the understanding of the zero range process complicated. However, the way dependencies spread is related to how the process $r_{t}$ grows. As we shall see, bounds on the speed of the infection front will provide bounds on the velocity with which dependencies can spread through space and time.

\newconstant{c:finite_speed}
\newconstant{c:positive_speed}

\par The first result we prove states that $r_{t}$ has finite velocity with high probability.

\begin{teo}\label{teo:finite_velocity}
For any $\rho>0$, there exist $v_{+}=v_{+}(\rho)>0$ such that, for all $L>0$,
\begin{equation}\label{eq:finite_speed}
\PP_{\rho}\left[\begin{array}{cl}
r_{t} \geq v_{+}t+L, \\
\text{for some } t \geq 0
\end{array}\right] \leq \useconstant{c:finite_speed}e^{-\useconstant{c:finite_speed}^{-1}\log^{\sfrac{5}{4}}L},
\end{equation}
for some positive constant $\useconstant{c:finite_speed}$ that depends only on the density $\rho$ and the rate function $g$.
\end{teo}

\par Our second result says that $r_{t}$ travels to the right with positive velocity.
\begin{teo}\label{teo:positive_velocity}
For any $\rho>0$, there exist $v_{-}=v_{-}(\rho)>0$ such that, for all $L>0$,
\begin{equation}\label{eq:positive_speed}
\PP_{\rho}\left[\begin{array}{cl}
r_{t} \leq v_{-} t-L, \\
\text{for some } t \geq 0
\end{array}\right] \leq \useconstant{c:positive_speed}e^{-\useconstant{c:positive_speed}^{-1}\log^{\sfrac{5}{4}}L},
\end{equation}
for some positive constant $\useconstant{c:positive_speed}$ that depends only on the density $\rho$ and the rate function $g$.
\end{teo}

\par The dependencies introduced by the zero range process require us to introduce new techniques as we cannot find simple renewal structures for the evolution of the infection front. An interesting open problem is proving that $v_{-}=v_{+}$ and thus deduce a law of large numbers for the infection front.

\par The process $r_{t}$ increases by one whenever an infected particle at position $r_{t}$ jumps to $r_{t}+1$. However, in order for $r_{t}$ to decrease, it is necessary that all infected particles at $r_{t}$ jump to $r_{t}-1$. This suggests that the process $r_{t}$ should have a tendency to go to the right. Turning this heuristics into a proof may seem easy at first sight. An indicative that this is not the case is the collection of works Ram{\'i}rez and Sidoravicius~\cite{rs}, Comets, Quastel and Ram{\'i}rez~\cite{cqr}, and B{\'e}rard and Ram{\'i}rez~\cite{br1} where a similar model is considered. There, healthy particles remain still until they become infected. Besides, infected particles move independently from each other. This independence assumption is of central importance, since it enables the introduction of well-behaved renewal structures and these techniques cannot be easily adapted for dependent systems, such as the zero range process. These works establish a law of large numbers, central limit theorem and large deviations for their models. Some of them also consider convergence of the environment as seen from the front, see~\cite{br1}.

\par Our theorem is a first step in understanding how influences spread in the zero range process: as a corollary, we obtain a correlation estimate for functions that depend only on sets that are far enough in space, see Proposition~\ref{prop:horizontal_decoupling}.

\bigskip

\paragraph{Proof Overview.} First, we prove that $r_{t}$ travels to the right with finite velocity. We use multiscale renormalisation to bound the probability of events where $r_{t}$ travels fast to the right at some fixed times. When we have a good bound for this fixed sequence of times, all the work remaining is to do an interpolation argument to conclude that the statement holds uniformly in time.

\par The proof of the second theorem is also based in multiscale renormalisation. However, we cannot apply the same argument using events where the front does not travel with some small but positive speed, since this would require a better understanding of more refined properties of the model. Instead, we use an alternative strategy inspired in Kesten and Sidoravicius~\cite{ks} considering a broad class of paths and prove that, for each of them, there is a positive fraction of time where at least two particles are close to the path. We observe that the front wave is one such path and, when two particles are close to it, there is a positive chance that these particles will meet in the front and produce a drift to the right. A central step in both proofs is the decoupling for the zero range process.

\bigskip

\paragraph{Decoupling.} A decoupling is an estimate on the dependence decay of functions whose supports are sufficiently far away and such results 
are interesting on their own. In the last few years, these types of estimates have proven to be a powerful tool when studying models that lack good mixing properties, see~\cite{baldasso_teixeira}, by the same authors, Benjamini and Stauffer~\cite{bs}, Hil\'{a}rio, den Hollander, Sirodavicius, Soares dos Santos and Teixeira~\cite{rwrw}, and Sidoravicius and Stauffer~\cite{ss}. Here, we prove a decoupling for the zero range process considering functions of the space-time that are far away in time.

\par We say that a function of the trajectories $f$ has support in a space-time box $B \subset \ZZ \times \RR_{+}$ if, for every pair of trajectories $\eta$ and $\bar{\eta}$,
\begin{equation}\label{eq:support_decoupling}
\eta_{t}(x)=\bar{\eta}_{t}(x) \text{ for all } (x,t) \in B \text{ implies } f(\eta)=f(\bar{\eta}).
\end{equation}
The partial order in space-time trajectories ($\eta \preceq \bar{\eta}$, if $\eta_{t}(x) \leq \bar{\eta}_{t}(x)$, for all $(x, t) \in \ZZ \times \RR_{+}$) allows us to say that a function $f: \NN_{0}^{\ZZ \times \RR_{+}} \to \RR$ is non-decreasing if
\begin{equation}\label{eq:non_decreasing}
\eta \preceq \bar{\eta} \text{ implies } f(\eta) \leq f(\bar{\eta}).
\end{equation}
Given two space-time boxes $B_{1}, B_{2} \subset \ZZ \times \RR_{+}$, the vertical distance between them is
\begin{equation}\label{eq:vertical_distance}
\dist_{V}=\inf \{|t-s|:(x,t) \in B_{1} \text{ and } (y,s) \in B_{2}\}.
\end{equation}

\newbigconstant{C:vertical_decoupling}
\newconstant{c:vertical_decoupling}

\par We now can state our decoupling, the main tool in the proof of Theorems~\ref{teo:finite_velocity} and~\ref{teo:positive_velocity}.
\begin{teo}\label{teo:vertical_decoupling}
Fix $\rho_{+}>0$. There exist positive constants $\usebigconstant{C:vertical_decoupling}=\usebigconstant{C:vertical_decoupling}(\rho_{+})$ and $\useconstant{c:vertical_decoupling}=\useconstant{c:vertical_decoupling}(\rho_{+})$ such that, for any two square boxes $B_{1}$ and $B_{2}$ of side-length $s >0$ satisfying
\begin{equation}
\dist_{V}=\dist_{V}(B_{1}, B_{2}) \geq \usebigconstant{C:vertical_decoupling},
\end{equation}
and any two non-decreasing functions of the space-time configurations $f_{1},f_{2}:\NN_{0}^{\ZZ \times \RR_{+}} \to [0,1]$ with respective supports in $B_{1}$ and $B_{2}$, we have, for any $\rho \in [0,\rho_{+}]$ and $\epsilon \in (0,1]$,
\begin{equation}\label{eq:decoupling_estimate}
\EE_{\rho}[f_{1}f_{2}] \leq \EE_{\rho+\epsilon}[f_{1}]\EE_{\rho+\epsilon}[f_{2}]+\useconstant{c:vertical_decoupling} \dist_{V}(\dist_{V}+s+1)e^{-\useconstant{c:vertical_decoupling}^{-1}\epsilon^{2}\dist_{V}^{\sfrac{1}{4}}}.
\end{equation}
\end{teo}

\begin{remark}
One can also take $f_{1}$ and $f_{2}$ non-increasing and assume that $\epsilon \in [-1,0)$. The proof carries out in the same way in this case.
\end{remark}

\begin{remark}
Observe that \eqref{eq:decoupling_estimate} is not a correlation estimate, since we need to add a sprinkling in order to have this bound on the error function. A question that rises naturally from the theorem above is if it is possible to take $\epsilon=0$, and avoid using sprinkling. In \cite{rwrw}, the authors consider a particle system composed by independent random walks evolving in discrete time. The continuous version of their model corresponds to a zero range process with rate funciton $g(n)=n$. They prove that the correlations do not decay as fast as the bound given in our theorem. In fact, Equation (2.11) from \cite{rwrw} provides an example where the correlations decay as $\dist_{V}^{-\sfrac{1}{2}}$.
\end{remark}

\par We believe this theorem may have many applications. It should help to understand processes that evolve on top of the zero range process. The random walk on top of the zero range process is an example of model that our decoupling should help to understand, by generalizing the results in~\cite{rwrw}. The challenge in this is to develop a renewal structure for such walk.

\bigskip

\paragraph{Proof overview.} For the proof of Theorem~\ref{teo:vertical_decoupling}, we couple two zero range processes with densities $\rho < \rho'$ in a way that, with high probability, the process with bigger density dominates the less dense process inside an interval for some fixed large time.

\par Let us describe the coupling. We will match each particle of the process that has smaller density, with a particle of the process with larger density $\rho'$, similarly to the couplings contructed in~\cite{bs} and~\cite{baldasso_teixeira}. This is done in a careful way so that each pair of particles is not far apart at time zero. The evolution of the process is constructed in a way that, when a pair of matched particles meets, they stay together from this time on. This single coupling attempt is not good enough, since the the probability that nearby particles avoid each other does not decay fast enough. To fix this, the matching is remade at some particular times and all the process starts again in order to match more particles.

\bigskip

\newbigconstant{c:horizontal_decoupling_1}
\newbigconstant{c:horizontal_decoupling_2}
\newconstant{c:horizontal_decoupling}

\par Finally, as an application, we also prove a decoupling for the zero range process considering functions of the space-time configurations whose supports are far away in space. Recall the definition of the vertical distance~\eqref{eq:vertical_distance} and consider the horizontal distance between the boxes $B_{1}$ and $B_{2}$
\begin{equation}\label{eq:horizontal_distance}
\dist_{H}=\inf \{|x-y|:(x,t) \in B_{1} \text{ and } (y,s) \in B_{2}\}.
\end{equation}

\begin{prop}\label{prop:horizontal_decoupling}
Fix $0< \rho_{-} \leq \rho_{+}< \infty$. There exist positive constants $\usebigconstant{c:horizontal_decoupling_1}=\usebigconstant{c:horizontal_decoupling_1}(\rho_{-}, \rho_{+})$, $\usebigconstant{c:horizontal_decoupling_2}=\usebigconstant{c:horizontal_decoupling_2}(\rho_{-}, \rho_{+})$ and $\useconstant{c:horizontal_decoupling}=\useconstant{c:horizontal_decoupling}(\rho_{-}, \rho_{+})$ such that, for any two square boxes $B_{1}$ and $B_{2}$ of side-length $s>0$ satisfying
\begin{equation}
\dist_{H} \geq \usebigconstant{c:horizontal_decoupling_1}(s+\dist_{V})+\usebigconstant{c:horizontal_decoupling_2},
\end{equation}
and any two functions of the space-time $f_{1},f_{2}:\NN_{0}^{\ZZ \times \RR_{+}} \to [0,1]$ with respective supports in $B_{1}$ and $B_{2}$, we have, for any $\rho \in [\rho_{-},\rho_{+}]$,
\begin{equation}\label{eq:decoupling_estimate_3}
\EE_{\rho}[f_{1}f_{2}] \leq \EE_{\rho}[f_{1}]\EE_{\rho}[f_{2}]+\useconstant{c:vertical_decoupling}e^{-\useconstant{c:vertical_decoupling}^{-1}\log^{\sfrac{5}{4}} d_{H}}
\end{equation}
\end{prop}

\bigskip

\paragraph{Related works.} There exists a rich literature concerning infection processes. Giacomelli~\cite{giacomelli} proves that, for our model, in the independent case, i.e., when the rate function $g$ equals the identity, the velocity of the infection wave is greater than one.

\par Jara, Moreno and Ram{\'i}rez~\cite{jmr} consider an infection process evolving on top of the exclusion process. Based on a regeneration argument, they prove a law of large numbers and central limit theorem for this model. 

\par Higher dimensional models have also been considered. Popov~\cite{popov} presents a detailed review of the so-called frog model. An extensive study of this model is conducted in Alves, Machado and Popov~\cite{amp, amp2}, and Alves, Machado, Popov and Ravishankar~\cite{ampr}.

\par In \cite{ks}, a model similar to ours is considered, but for any dimension: particles evolve as independent random walks and only the origin begins infected. In this case, they prove a shape theorem with similar techniques as the ones we use in the proof of Theorem~\ref{teo:positive_velocity}.

\par Decoupling estimates using sprinklings were also introduced in the context of random interlacements by Sznitmann~\cite{s}, and Popov and Teixeira~\cite{pt}. Later on, they were used for studying other types of random processes, such as independent Brownian motions by Peres, Sinclair, Sousi and Stauffer~\cite{psss}, and Stauffer~\cite{stauffer}. More recently, in~\cite{rwrw},~\cite{baldasso_teixeira},~\cite{psss} and~\cite{stauffer}, conservative particle systems where considered. In~\cite{rwrw}, the authors prove a decoupling for systems composed by particles performing independent random walks. Their techniques are similar to ours, but the coupling they obtain is of different nature, using the results of~\cite{pt}. Our techniques are somewhat similar to the ones in~\cite{bs},~\cite{baldasso_teixeira} and~\cite{stauffer}, where a decoupling for the exclusion process is proved. 

\begin{remark}
Contrary to~\cite{ks}, we restrict ourselves to the one-dimensional case, but we believe that similar statements can be made for larger dimensions. For a proof, the renormalisation we present here should apply, however one needs to be more careful with the decoupling. In larger dimensions, when the coupling used to prove Theorem~\ref{teo:vertical_decoupling} is constructed, the probability that two particles never meet does not converge to zero. Nevertheless, we believe this can be solved by using more careful estimates of the Green function. We leave this as an interesting direction of further investigation, but refrain to pursuing it here due to size constrains.
\end{remark}

\bigskip

\textbf{Structure of the paper.} Section~\ref{sec:zrp} is devoted to the precise definition of the zero range process and presents its graphical construction. The proof of decoupling for the zero range process can be found in Section~\ref{sec:vertical_decoupling}. In Section~\ref{sec:infection}, we introduce the infection process and prove some lemmas about it. The proof of Theorem~\ref{teo:finite_velocity} can be found in Section~\ref{sec:finite_vel}. Finally, Section~\ref{sec:positive_vel} contains the proof of Theorem~\ref{teo:positive_velocity}.

\bigskip

\textbf{Acknowledgements.} The authors thank Milton Jara for valuable discussions on the initial stages of the work. RB thanks  FAPERJ grant E-26/202.231/2015 for financial support and IMPA for the hospitality during the development of this work. AT thanks CNPq grants 306348/2012-8 and 478577/2012-5 and FAPERJ grant 202.231/2015 for financial support.

\section{The zero range process}\label{sec:zrp}
~

\par In this section we define and recall some properties of the zero range process.

\par The zero range process in $\ZZ$ is a system where particles interact only when they are at the same site. This interaction alters the jump rate of a particle according to the number of particles that share the site.

\par Fix a non-negative function $g:\NN_{0} \to \RR_{+}$ with $g(0)=0$ and a translation invariant transition probability $p(\cdot,\cdot)$ on $\ZZ$. The zero range process with rate function $g$, transition probability $p$ and initial state $\eta_{0} \in \ZZ^{\NN_{0}}$ is the particle system on $\NN_{0}^{\ZZ}$ with infinitesimal generator given by
\begin{displaymath}
Lf(\eta)=\sum_{x \in \ZZ} g(\eta(x)) \sum_{y \in \ZZ} p(x,y)\Big[f(\eta^{x,y})-f(\eta)\Big],
\end{displaymath}
where $\eta^{x,y}$ is the configuration obtained from $\eta$ by taking one particle from site $x$ and placing it at site $y$ and $f$ is any bounded local function. We will soon provide classical conditions for the existence of the process.

\par In this process, particles interact only when they are at the same site. The interaction is given by the function $g$ that controls the jump rate.

\par We are interested in the case where $p$ is the nearest-neighbor symmetric transition probability, $p(0,1)=p(0,-1)=\sfrac{1}{2}$, and $g$ satisfies~\eqref{eq:g_condition}.

\bigskip

\par For $\phi \in \RR_{+}$, consider the product measure with marginals $\nu_{\phi}$ given by
\begin{equation}\label{eq:inv_measure}
\nu_{\phi}(k)=\frac{1}{Z(\phi)}\frac{\phi^{k}}{g(k)!}, \text{ for all } k \in \NN_{0},
\end{equation}
where $g(k)!=g(k)\cdot g(k-1) \cdots g(1)$, $g(0)!=1$ and $Z(\phi)$ is a normalizing constant:
\begin{equation}\label{eq:normalization}
Z(\phi)=\sum_{k=0}^{\infty}\frac{\phi^{k}}{g(k)!}.
\end{equation}
Observe that the lower bound in Assumption~\eqref{eq:g_condition} implies that, for all $\phi \in \RR_{+}$, $Z(\phi)< \infty$ and hence these probability measures are well-defined. These probabilities measures are invariant and compose the collection of invariant measures for the zero range process that we consider. We remark however that these are not a complete set of invariant measures for the zero range process, as proved in~\cite{andjel}.

\begin{remark}
We will use a slight abuse of notation, by denoting the product measure and its marginals by the same symbols.
\end{remark}

\par In general, the parameter $\phi$ is not the density of the process. In fact, for the measure $\nu_{\phi}$, the expected number of particles in each site is given by
\begin{equation}\label{eq:density}
R(\phi)=\frac{1}{Z(\phi)}\sum_{k=0}^{\infty}\frac{k\phi^{k}}{g(k)!} = \phi \frac{Z'(\phi)}{Z(\phi)}.
\end{equation}
The function $R:\RR_{+} \to \RR_{+}$ is an increasing bijection. This implies that we can parametrize the measures in \eqref{eq:inv_measure} by density:
\begin{equation}\label{eq:inv_measure_density}
\mu_{\rho}=\nu_{R^{-1}(\rho)}.
\end{equation}
We refer to Section 2.3 of~\cite{kipnis_landim} for further information about these measures.

\bigskip

\par Theorem 1.4 from~\cite{andjel} implies that the process starting from any measure $\mu_{\rho}$ exists with probability one. Besides, the stationary process is reversible, which allows us to define it for all real times, and obtain $(\eta_{s})_{s \in \RR}$.

\bigskip

\par In order to prove our decoupling, an important ingredient is concentration of the invariant measures. This is the content of the next proposition.

\begin{prop}\label{prop:concentration}
Let $(X_{k})_{k=1}^{n}$ be a collection of independent random variables with distribution $\mu_{\rho}$. For any $\epsilon \in (0,1]$,
\begin{equation}\label{eq:concentration_upper_bound}
\PP_{\rho}\left[\sum_{k=1}^{n}X_{k} \geq (\rho+\epsilon)n\right] \leq e^{-c(\rho)\epsilon^{2}n},
\end{equation}
and
\begin{equation}\label{eq:concentration_lower_bound}
\PP_{\rho}\left[\sum_{k=1}^{n}X_{k} \leq (\rho-\epsilon)n\right] \leq e^{-c(\rho)\epsilon^{2}n},
\end{equation}
where $c(\rho)$ is a constant that depends on $\rho$ and is uniformly bounded on compact intervals of $[0, \infty)$.
\end{prop}

\par We defer the proof of this proposition to the Appendix.

\newconstant{c:large_deviation}

\par For future reference, we introduce a constant $\useconstant{c:large_deviation}>0$ satisfying
\begin{equation}
\frac{Z(eR^{-1}(\rho))}{Z(R^{-1}(\rho))}e^{-\useconstant{c:large_deviation}\rho} \leq 1, \,\, \text{ for all } \rho \in [0,\rho_{+}].
\end{equation}
This choice of constant gives that, for all $\rho \in [0, \rho_{+}]$,
\begin{equation}\label{eq:large_deviation}
\PP_{\rho}\left[\sum_{k=1}^{n}X_{k} \geq \useconstant{c:large_deviation}\rho n+t\right] \leq \left[\frac{Z(eR^{-1}(\rho))}{Z(R^{-1}(\rho))}e^{-\useconstant{c:large_deviation}\rho}\right]^{n}e^{-t} \leq e^{-t}.
\end{equation}

\subsection{A graphical construction for the zero range process}\label{subsec:graphical}
~
\par This subsection is devoted to a graphical construction for the zero range process. This construction will be used in the coupling presented in Subsection~\ref{subsec:coupling}.

\par In this construction of the process, every site $x \in \ZZ$ has an associated Poisson point process $\mathcal{P}(x)$ that will control the jumps from $x$. The points of the process have the form $(t,n,u,h)$, where $t$ describes the time of a jump, $n$ describes the height of the particle that is moved, $u$ is an uniformly distributed auxiliary random variable that will help in controlling the jump rate, and $h$ is the direction of the jump. Each Poisson point process takes values in $\RR_{+} \times \NN \times [0,1] \times \{-1,+1\}$ and has intensity measure $\Gamma_{+} \lambda \otimes \mu \otimes \lambda \otimes \sfrac{1}{2}(\delta_{-1}+\delta_{+1})$, where $\Gamma_{+}$ is the constant defined in \eqref{eq:g_condition}, $\mu$ is the counting measure and $\lambda$ is the usual Lebesgue measure.

\par The evolution is set in the following way. Suppose that, at some site $x$, we have a point from the Poisson point process of the form $(t, n, u, h)$ and that the configuration, at this time, has at least $n$ particles at $x$. The particle at height $n$ will perform a jump directed according to $h$ if
\begin{equation}\label{eq:jump_condition}
u \leq \frac{g(n)-g(n-1)}{\Gamma_{+}}.
\end{equation}
If the jump is allowed, all particles that are above the selected particle at site $x$ go down one position and the particle that jumps lands at the top of its next pile. Whenever \eqref{eq:jump_condition} does not hold or the pile contains less then $n$ particles, the jump is simply suppressed.

\par This construction allows us to bound the probability that a site has many particles at some time in $[0,t]$, as stated in the next lemma.

\newconstant{c:many_particles}

\begin{lemma}\label{lemma:many_particles}
Given $\rho_{+}>0$, there exists $\useconstant{c:many_particles}=\useconstant{c:many_particles}(\rho_{+})>0$ such that, for all $\rho \in [0, \rho_{+}]$, if $A(u,t)=(2u+4\Gamma_{+} t)(\rho+1)+1$, then
\begin{equation}\label{eq:many_particles}
\PP_{\rho}\left[\begin{array}{cl}
\eta_{s}(0) \geq A(u, t), \\ \text{for some } s \in [0,t]
\end{array}\right] \leq \useconstant{c:many_particles}(t+1)e^{-\useconstant{c:many_particles}^{-1}u}.
\end{equation}
\end{lemma}

\par This lemma is not sharp and can be regarded as a rough estimate that will be used to obtain better bounds later in the text. The quantity $A(u,t)$ is chosen so that we can use concentration of the invariant measure in a large interval around the origin. The strategy of the proof is to observe that if the event in the lemma holds, either some large interval has many particles or some particle reaches the origin from very far away.

\begin{proof}
Let $B$ be the event described in the lemma. Observe that, if $B$ holds, either the interval $J_{t}=[-\lfloor 2 \Gamma_{+} t +u \rfloor,\lfloor 2 \Gamma_{+} t+u \rfloor]$ contains more that $A(u,t)-1$ particles at time zero or some particle that started outside $J_{t}$ reaches zero before time $t$. Let $Y_{t} \sim \poisson(\Gamma_{+} t)$. Since each particle jumps at most $\poisson(\Gamma_{+} t)$ times, Proposition~\ref{prop:concentration} and Equation \eqref{eq:large_deviation} can be used to bound
\begin{equation}\label{eq:bound_A_many_particles}
\begin{split}
\PP[B] & \leq \PP[\sum_{k \in J_{t}}\eta_{0}(k) \geq A(u,t)-1]\\ 
& \quad + 2\sum_{y \geq 2 \Gamma_{+} t+u} \PP[\eta_{0}(y) \geq \useconstant{c:large_deviation}\rho+y]+(\useconstant{c:large_deviation}\rho + y)\PP[Y_{t} \geq y] \\
& \leq e^{-\useconstant{c:many_particles}^{-1}u}+2\sum_{y \geq 2 \Gamma_{+} t+u}e^{-y}+(\useconstant{c:large_deviation}\rho +y)e^{-\frac{y}{3}} \leq \useconstant{c:many_particles}(t+1)e^{-\useconstant{c:many_particles}^{-1}u},
\end{split}
\end{equation}
concluding the proof.
\end{proof}

\section{Vertical decoupling}\label{sec:vertical_decoupling}
~
\par In this section, we construct the main step towards the proof of our main theorem. We prove a vertical decoupling for the zero range process.

\newconstant{c:coupling_zrp}
\newbigconstant{C:coupling_zrp}

\par The next proposition is a central tool used in the proof of Theorem~\ref{teo:vertical_decoupling}. It provides a coupling between two zero range processes with densities $\rho$ and $\rho+\epsilon$ in a way that the process with larger density dominates the other in a fixed interval for some large time $t$.
\begin{prop}\label{prop:coupling}
Given $\rho_{+}>0$, there exist positive constants $\useconstant{c:coupling_zrp}=\useconstant{c:coupling_zrp}(\rho_{+})$ and $\usebigconstant{C:coupling_zrp}=\usebigconstant{C:coupling_zrp}(\rho_{+})$ such that, for any $t \geq \usebigconstant{c:coupling_zrp}$, interval $I \subset \ZZ$, density $\rho \in [0,\rho_{+}]$ and $ \epsilon \in (0,1]$, there exists a coupling between two zero range processes $(\eta_{s})_{s \in \RR}$ and $(\bar{\eta}_{s})_{s \geq 0}$ such that
\begin{enumerate}
\item[1.] $(\eta_{s})_{s \in \RR}$ has density $\rho$ and $(\bar{\eta}_{s})_{s \geq 0}$ has density $\rho +\epsilon$;
\item[2.] $(\bar{\eta}_{s})_{s \geq 0}$ is independent from $(\eta_{s})_{s \leq 0}$;
\item[3.]
\begin{equation}\label{eq:bad_coupling}
\PP\left[
\begin{array}{cl}
\text{there exists } x \in I \\ \text{such that } \eta_{t}(x) > \bar{\eta}_{t}(x)
\end{array}
\right] \leq \useconstant{C:coupling_zrp} t (|I|+t)e^{-\useconstant{C:coupling_zrp}^{-1}\epsilon^{2}t^{\sfrac{1}{4}}}.
\end{equation}
\end{enumerate}
\end{prop}

\par  We postpone the proof of this proposition to the next subsection. Assuming its validity, we are in position to prove the decoupling for the zero range process.

\begin{proof}[Proof of Theorem~\ref{teo:vertical_decoupling}]
In the proof, it is useful to keep Figure~\ref{fig:boxes_vertical} in mind. Without loss of generality, we assume that the boxes have the form
\begin{align*}
B_{1}=\left[-\sfrac{s}{2}, \sfrac{s}{2}\right]\times[-s,0], \\
B_{2}=\left[a, a+s\right]\times[\dist_{V},\dist_{V}+s],
\end{align*}
where $\sfrac{s}{2}$ and $a$ are positive integer numbers.

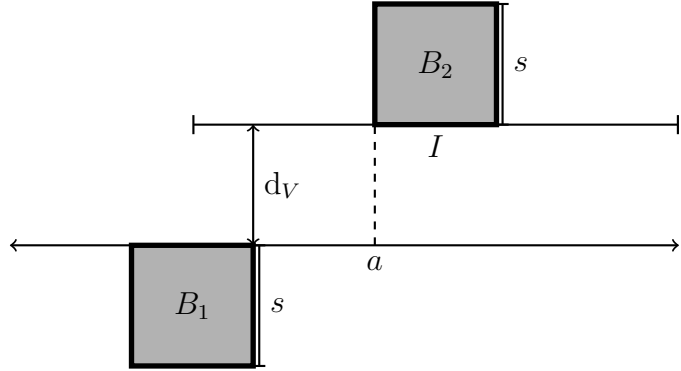
\begin{figure}[h]\label{fig:boxes_vertical}
\begin{center}
\begin{tikzpicture}[scale=0.8]
\draw[<->, thick]  (-3,0) -- (8,0);

\fill [black!30!](-1,-2) rectangle (1,0);
\draw [line width=2pt] (-1,-2) rectangle (1,0);
\node at (0,-1) {$B_{1}$};

\fill [black!30!](3,2) rectangle (5,4);
\draw [line width=2pt] (3,2) rectangle (5,4);
\node at (4,3) {$B_{2}$};

\draw[|-|, thick](0,2)--(8,2);
\node[below] at (4,2) {$I$};

\draw[<->, thick]  (1,0) -- (1,2);
\node[right] at (1,1) {d$_{V}$};

\draw[thick] (5.1,2)--(5.1,4);
\draw[thick] (5,2)--(5.2,2);
\draw[thick] (5,4)--(5.2,4);
\node[right] at (5.1,3) {$s$};

\draw[thick] (1.1,0)--(1.1,-2);
\draw[thick] (1,0)--(1.2,0);
\draw[thick] (1,-2)--(1.2,-2);
\node[right] at (1.1,-1) {$s$};

\draw[dashed, thick]  (3,0) -- (3,2);
\node[below] at (3,0) {$a$};
\end{tikzpicture}
\caption{The boxes $B_{1}$, $B_{2}$ and the interval $I$.}
\end{center}
\end{figure}

Let $I=[-\lceil 2 \Gamma_{+} s + \dist_{V}\rceil +a, a+s +\lceil 2 \Gamma_{+} s+ \dist_{V}\rceil]$ and  define the event
\begin{equation}
E=\left\{\begin{array}{cl}
\text{some particle of } \eta \text{ is outside } I \\ \text{ at time } \dist_{V} \text{ and enters the box } B_{2}
\end{array}
\right\}.
\end{equation}

If $\dist_{V} \geq \usebigconstant{C:coupling_zrp}$, we can use the coupling of Proposition~\ref{prop:coupling} with the interval $I$. Define the bad event for the coupling
\begin{equation}
F=\left\{
\begin{array}{cl}
\text{there exists } x \in I \\ \text{such that } \eta_{\dist_{V}}(x) > \bar{\eta}_{\dist_{V}}(x)
\end{array}
\right\}.
\end{equation} 

Let $\EE$ denote the expectation with respect to the coupling measure provided by Proposition~\ref{prop:coupling}. The fact that the functions $f_{1}$ and $f_{2}$ are non-decreasing and Markov's property can be used to obtain
\begin{equation}\label{eq:V_decoupling_bound_1}
\begin{split}
\EE_{\rho}[f_{1}f_{2}]& \leq \EE[f_{1}(\eta) f_{2}(\eta)(\charf{E^{c}\cap F^{c}}+\charf{E}+\charf{F})]\\
& \leq \EE[f_{1}(\eta)f_{2}(\bar{\eta})\charf{E^{c}\cap F^{c}}] +\PP[E]+\PP[F]\\
& \leq \EE[f_{1}(\eta)f_{2}(\bar{\eta})] +\PP[E]+\PP[F]\\
& = \EE[f_{1}(\eta)\EE[f_{2}(\bar{\eta})|\{\eta_{s}:s \leq 0\}]] +\PP[E]+\PP[F]\\
& = \EE[f_{1}(\eta)\EE[f_{2}(\bar{\eta})]] +\PP[E]+\PP[F]\\
& \leq \EE_{\rho}[f_{1}]\EE_{\rho+\epsilon}[f_{2}]+\PP[E]+\PP[F].
\end{split}
\end{equation}

Proposition~\ref{prop:coupling} implies, by possibly increasing constants, that
\begin{equation}\label{eq:V_decoupling_bound_2}
\PP[F] \leq \useconstant{c:coupling_zrp} \dist_{V}(\dist_{V}+s)e^{-\useconstant{c:coupling_zrp}^{-1}\epsilon^{2}\dist_{V}^{\sfrac{1}{4}}},
\end{equation}
It remains to bound the probability of $E$. Here, we apply the same ideas from the proof of Lemma~\ref{lemma:many_particles}. We use symmetry and the fact that, in order for a particle that is at site $y+\lceil 2 \Gamma_{+} s +\dist_{V} \rceil$, with $y \geq 0$, at time $\dist_{V}$ to enter $B_{2}$, it is necessary for it to jump at least $y+\lceil 2 \Gamma_{+} s +\dist_{V} \rceil$ times before time $\dist_{V}+s$. Since the number of jumps a particle performs between times $\dist_{V}$ and $\dist_{V}+s$ is bounded by a random variable $X \sim \poisson(\Gamma_{+} s)$ , we obtain
\begin{equation}\label{eq:V_decoupling_bound_3}
\begin{split}
\PP[E] & \leq 2\sum_{y \geq 0} \PP[\eta_{\dist_{V}}(y+\lceil 2 \Gamma_{+} s +\dist_{V} \rceil) \geq \useconstant{c:large_deviation}\rho+\dist_{V}+1+y] \\
& \qquad + 2\sum_{y \geq0}(\useconstant{c:large_deviation}\rho+\dist_{V}+1 + y)\PP[X \geq y+\lceil 2 \Gamma_{+} s +\dist_{V} \rceil] \\
& \leq 2\sum_{y \geq 0}e^{-y-\dist_{V}}+(\useconstant{c:large_deviation}\rho+\dist_{V}+1+y)e^{-y-\dist_{V}} \\
& \leq c(\dist_{V}+1)e^{-\dist_{V}}.
\end{split}
\end{equation}

Combining Equations \eqref{eq:V_decoupling_bound_1}, \eqref{eq:V_decoupling_bound_2}, \eqref{eq:V_decoupling_bound_3}, and possibly changing constants concludes the proof.
\end{proof}

\subsection{The coupling}\label{subsec:coupling}
~
\par This subsection is devoted to the construction of the coupling stated in Proposition~\ref{prop:coupling}. We begin this subsection with an informal description of it.

\par Fix two initial independent configurations $\eta_{0} \sim \mu_{\rho}$ and $\bar{\eta}_{0} \sim \mu_{\rho+\epsilon}$. The strategy is to match the particles of the configuration $\eta_{0}$ to particles of the configuration $\bar{\eta}_{0}$. Once this matching is constructed, we set the joint evolution of the pair $(\eta_{s},\bar{\eta}_{s})$.

\par The processes evolve in such a way that, if two matched particles share at any time the same site, they keep moving together. This will help to assure that $\eta_{t}(x) \leq \bar{\eta}_{t}(x)$, for every $x \in I$, with high probability.

\par The correct construction of the matching is important to ensure that each pair meets fast enough with large probability. This is done by restricting the distance between two particles that are matched.

\par For the evolution, we use the matching and two independent copies of the graphical construction presented in Subsection~\ref{subsec:graphical}. This will help to evolve both processes in a way that particles that have met their pairs do not disturb the particles that still did not and hence do not decrease the probability of the meeting event.

\par However, this construction is not enough to obtain the desired result, since, as we will see, the decay of the probability that two matched particles do not meet is related with the probability that a random walk does not reach zero which does not decay fast enough. We improve this bound by remaking the matching at some fixed times, allowing particles to have new pairs and new chances to meet.

\par We now begin the construction of the coupling. We follow the lines of~\cite{baldasso_teixeira}.

\begin{remark}
All constants that appear from now on are uniformly bounded for any $\rho \in [0,\rho_{+}]$ and may depend also on $\Gamma_{-}$ and $\Gamma_{+}$. We will omit these dependencies.
\end{remark}

\par The first step is to fix an interval $H$ that contains $I= [a,b]$. In our case, we set $H=[a-\lceil 3 \Gamma_{+} t\rceil, b+ \lceil 3 \Gamma_{+} t\rceil]$. This choice allows us to easily bound the probability that a particle that is outside $H$ reaches $I$ before time $t$. Now, split $H$ into a collection of subintervals $(I_{j})_{j=1}^{N}$. We will assume that all intervals $I_{j}$ have the same size $L=\lfloor t^{\sfrac{1}{4}}\rfloor$. It is possible to assure this if we increase the size of $H$ by at most $L$. Besides, the number of intervals $N$ is clearly bounded by $|H|$.

\newconstant{c:concentration}

\par For any configuration $\bar{\eta}$, denote by $\sigma_{j}(\bar{\eta})=\sum_{x \in I_{j}}\bar{\eta}(x)$ the number of particles of $\bar{\eta}$ inside the interval $I_{j}$. We have the following claim.
\begin{claim}\label{claim:concentration}
If $\eta \sim \mu_{\rho}$ and $\bar{\eta} \sim \mu_{\rho+\epsilon}$, with $\epsilon \in (0,1]$, then
\begin{equation}\label{eq:concentration}
\PP[\exists \, j \leq N: \sigma_{j}(\eta) > \sigma_{j}(\bar{\eta})] \leq 2Ne^{-\useconstant{c:concentration} \epsilon^{2}t^{\sfrac{1}{4}}},
\end{equation}
even if the configurations are not independent.
\end{claim}

\begin{proof}
It follows directly from Proposition~\ref{prop:concentration}.
\end{proof}

\par We now sample independently two configurations $\eta_{0} \sim \mu_{\rho}$ and $\bar{\eta}_{0} \sim \mu_{\rho+\epsilon}$ and assume that the event in~\eqref{eq:concentration} does not hold.

\par The next step is to match the two configurations inside each of the intervals of the partition. In this matching, each particle of the configuration $\eta_{0}$ that lies inside the interval $I_{j}$ will be paired to a particle of the configuration $\bar{\eta}_{0}$ that is inside the same interval, but this construction is done in a special way. In the first step, for each site $x \in I_{j}$, we match the largest number possible of particles of $\eta_{0}$ at $x$ to particles of $\bar{\eta}_{0}$ that are at the same site (see Figure~\ref{fig:pairing}). Once this is done we can finish. There are many ways to match the remaining particles in a deterministic way. We fix an arbitrary algorithm from now on.

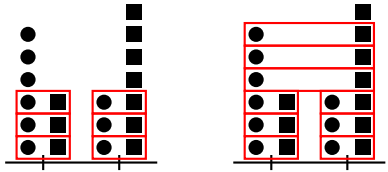
\begin{figure}[h]\label{fig:pairing}
\begin{center}
\begin{tikzpicture}

\draw[thick] (-2.5,0)--(-0.5,0);
\draw[thick] (0.5,0)--(2.5,0);
\foreach \x in {-2, -1, 1, 2}
		\draw[thick](\x,-0.1)--(\x,0.1);

\foreach \x in {-2,-1}
	{
		\foreach \y in {0,0.3,0.6}
			{
				\fill[black] (\x-0.2,0.2+\y) circle (0.1);
				\fill[black] (\x+3-0.2,0.2+\y) circle (0.1);
			}
	}
\foreach \y in {0.9,1.2,1.5}
	{
		\fill[black] (-2.2,0.2+\y) circle (0.1);
		\fill[black] (0.8,0.2+\y) circle (0.1);
	}

\foreach \x in {-2,-1}
	{
		\foreach \y in {0,0.3,0.6}
			{
				\fill[black] (\x+0.1,\y+0.1) rectangle (\x+0.3,\y+0.3);
				\fill[black] (\x+3+0.1,\y+0.1) rectangle (\x+3+0.3,\y+0.3);
			}
	}
\foreach \y in {0.9,1.2,1.5,1.8}
	{
		\fill[black] (-0.9,\y+0.1) rectangle (-0.7,\y+0.3);
		\fill[black] (2.1,\y+0.1) rectangle (2.3,\y+0.3);
	}
			
\foreach \x in {-2,-1,1,2}
	{
		\foreach \y in {0,0.3,0.6}
			{
				\draw[red, thick] (\x-0.35,\y+0.35) rectangle (\x+0.35, \y+0.05);
			}
	}
\foreach \y in {0.9,1.2,1.5}
	{
		\draw[red, thick] (1-0.35,\y+0.35) rectangle (2+0.35, \y+0.05);
	}

\end{tikzpicture}
\caption{The construction of a matching between two configurations. Balls represent the process $\eta$ and squares represent the configuration $\bar{\eta}$. First, pair as many particles of $\eta$ to particles of $\bar{\eta}$ as possible that are in the same site, and then complete the construction in an arbitrary deterministic way.}
\end{center}
\end{figure}

\par Once we have this matching, it is time to set the evolution for positive times. We proceed as follows. Let $\mathscr{P}_{1}=(\mathcal{P}_{1}(x))_{x \in \ZZ}$ and $\mathscr{P}_{2}=(\mathcal{P}_{2}(x))_{x \in \ZZ}$ be two independent copies of the graphical construction described in Subsection~\ref{subsec:graphical}. We use the clocks from $\mathscr{P}_{1}$ to evolve the process $(\bar{\eta}_{s})_{s \geq 0}$. On the other hand, the process $(\eta_{s})_{s \geq 0}$ will alternate between both constructions: if a particle of $\eta$ has met its pair, it uses the clocks from $\mathscr{P}_{1}$. Otherwise, it moves with the graphical construction $\mathscr{P}_{2}$.

\par Observe however that if a particle always jumps to the top of its new pile, then it is not necessarily true that particles that meet jump together as the two piles could have different heights. This is fixed by updating the order in the piles after each jump. More precisely, if a pair of matched particles jumps together, they will land at the bottom of the pile. Moreover, when a particle jumps alone, it will look for its matching particle at the next pile: if the particle and its pair are at the same site, they will both move to the bottom of the pile. Otherwise, the new particle will land on the top of its new corresponding pile.

\par This construction ensures that, if two particles have met, they remain together, and allows for pairs of particles that did not meet to do so.

\par Since the process $(\bar{\eta}_{s})_{s \geq 0}$ follows the original graphical construction up to changing heights of particles in the piles, it clearly behaves like a zero range process. Besides, it is independent from $(\eta_{0})_{s \leq 0}$, since it depends only on $\bar{\eta}_{0}$ and on $\mathscr{P}_{1}$. From now on, we will only consider the process $\eta$ restricted to positive times. It remains to prove that the $(\eta_{s})_{s \geq 0}$ is also a zero range process.

\begin{claim}\label{claim:coupling}
The process $(\eta_{s})_{s \geq 0}$ is a zero range process.
\end{claim}

\begin{proof}
Consider first the case when $\eta_{0}$ has finitely many particles. In this case, simply observe that the only modification we perform is changing the heights of the particles and choosing which of the two possible clocks these particles use.  This implies that our dynamics is a zero range process.

The case when $\eta_{0}$ has infinitely many particles is more delicate and we treat it in Appendix~\ref{apdx:claim}.
\end{proof}

\par Now that we have the main part of the coupling, we work the details in order to obtain the bound in~\eqref{eq:bad_coupling}. We introduce the \textbf{matching times} $(t_{k})_{k=0}^{\lfloor t^{\sfrac{1}{4}} \rfloor}$ defined by $t_{k}=kt^{\sfrac{3}{4}}$. At these times, the matching is remade preserving the couples already formed. This procedure will help the particles that still have not found their pairs by giving them new ones that are hopefully closer to them than their old partners were.

\par We now need to bound the probability that some particle that lies inside the interval $I$ at time $t$ did not find a couple during the time interval $[0,t]$. Let
\begin{equation}
A=\left\{\begin{array}{cl}
\text{there exists a particle from } \eta \text{ that is inside $I$} \\ \text{at time $t$} \text{ and did not find a couple in any of its attempts}
\end{array}\right\}.
\end{equation}

\par To bound the probability of $A$, we begin by bounding the probability of some bad events. The first event we introduce is related to the possibility that some particle that ends up in the interval $I$ at time $t$ does not find a couple because it is outside the interval $H$ at some time where the matching is remade. We consider
\begin{equation}\label{eq:B}
B=\left\{\begin{array}{cl}
\text{there exists a particle from } \eta \text{ that spends time} \\ \text{outside $H$ and ends up inside $I$ at time $t$}
\end{array}\right\}.
\end{equation}
The next event deals with the possibility that, for some matching time, it is not possible to construct the matching. Recall that $\sigma_{j}(\bar{\eta})=\sum_{x \in I_{j}}\bar{\eta}(x)$ and define the event
\begin{equation}
C=\left\{\begin{array}{cl}
\text{there exist a matching time $t_{k}$ and $j \in [N]$} \\ \text{such that } \sigma_{j}(\eta_{t_{k}}) > \sigma_{j}(\bar{\eta}_{t_{k}})
\end{array}\right\}.
\end{equation}

\par The bound in the probability of $C$ follows from Claim~\ref{claim:concentration} and union bound. We obtain
\begin{equation}\label{eq:bound_C}
\PP[C] \leq 2(t^{\sfrac{1}{4}}+1)Ne^{-\useconstant{c:concentration} \epsilon^{2}t^{\sfrac{1}{4}}}.
\end{equation}

The bound on the probability of $B$ is more delicate, and we state it as a claim.

\newconstant{c:bound_B}
\begin{claim}
There exists a constant $\useconstant{c:bound_B}>0$ such that, if $t$ is large enough,
\begin{equation}\label{eq:bound_B}
\PP[B] \leq \useconstant{c:bound_B}t^{2}e^{-\useconstant{c:bound_B}^{-1} t}.
\end{equation}
\end{claim}

\begin{proof}
Denote by $x$ the leftmost site at the right of $H$. By symmetry, we only need to bound the probability that there exists a particle that spends some time at $x$ and is inside $I$ at time $t$.

To bound the probability of $B$, let $A(t,t)$ be as in Lemma~\ref{lemma:many_particles} and consider
\begin{equation}
\tilde{A}=\left\{\begin{array}{cl}
\eta_{s}(x) \geq A(t,t), \\ \text{for some } s \in [0,t]
\end{array}\right\},
\end{equation}
and
\begin{equation}
\tilde{B}=\left\{\begin{array}{cl}
\text{more than $3\Gamma_{+} A(t,t) t$ clocks} \\ \text{ring at $x$ before time $t$}
\end{array}\right\}.
\end{equation}

Since the number of jumps a fixed particle performs before time $t$ is bounded by a random variable $X \sim \poisson(\Gamma_{+} t)$, union bounds gives
\begin{equation}\label{eq:bound_B_begining}
\begin{split}
\PP[B] & \leq 2\left(\PP[\tilde{A}]+\PP[\tilde{B} \cap \tilde{A}^{c}]+3 \Gamma_{+} A(t,t)t\PP[X \geq 3 \Gamma_{+} t]\right) \\
& \leq 2\left(\PP[\tilde{A}]+\PP[\tilde{B} \cap \tilde{A}^{c}]+3 \Gamma_{+} A(t,t)t e^{-\Gamma_{+} t}\right).
\end{split}
\end{equation}

It remains to bound the probability of the events $\tilde{A}$ and $\tilde{B} \cap \tilde{A}^{c}$. For the later, observe that, in $\tilde{A}^{c}$, the number of clocks that ring at site $x$ before time $t$ is dominated by a Poisson random variable with mean $\Gamma_{+} A(t,t)t$. This implies
\begin{equation}\label{eq:bound_B_tilde}
\PP[\tilde{B} \cap \tilde{A}^{c}] \leq \PP[\poisson(\Gamma_{+} A(t,t) t) \geq 3 \Gamma_{+} A(t,t)t] \leq e^{-\Gamma_{+} t}.
\end{equation}

A bound on the probability of $\tilde{A}$ is obtained in Lemma~\ref{lemma:many_particles}. Combining Equations~\eqref{eq:many_particles},~\eqref{eq:bound_B_begining},~\eqref{eq:bound_B_tilde} and increasing, if necessary, the value of $t$, we conclude the claim.
\end{proof}

\par Assume we are in the event $B^{c}\cap C^{c}$. The next step is to bound the probability that a fixed particle that lies inside $I$ at time $t$ does not find a couple.

\par First, observe that, since particles of both process move faster than random walks with jump rate $\Gamma_{-}$, the probability that two particles do not meet between two matching times is at most the probability that a random walk with jump rate $2\Gamma_{-}$ and starting somewhere in the interval $[0,L]$ do not reach zero before time $t_{1}$. Since the initial distance between the pair is at most $L$, if $(X_{s})_{s \geq 0}$ is a random walk that jumps with rate one, standard heat-kernel bounds allows us to estimate
\begin{equation}
\begin{split}
\PP & \left[ \begin{array}{cl}
\text{a fixed pair matched of particles} \\ \text{do not meet before time $t^{\sfrac{3}{4}}$}
\end{array}
\right]  \leq \max_{0 \leq k \leq L}  \PP_{k}\left[\inf_{u\leq 2\Gamma_{-} t^{\sfrac{3}{4}}}X_{u} >0 \right] \\
& = \max_{0 \leq k \leq L} \PP_{0}\left[\sup_{u\leq 2\Gamma_{-} t^{\sfrac{3}{4}}}X_{u} <k \right] = \PP_{0}\left[\sup_{u\leq 2\Gamma_{-} t^{\sfrac{3}{4}}}X_{u} <L \right] \\
& = 1-\PP_{0}\left[\sup_{u\leq 2\Gamma_{-} t^{\sfrac{3}{4}}}X_{u} \geq L \right] \\
& \leq 1-2\PP_{0}\left[X_{2\Gamma_{-} t^{\sfrac{3}{4}}} > L \right] = \PP_{0}\left[|X_{2\Gamma_{-} t^{\sfrac{3}{4}}}| \leq L \right] \\
& = \sum_{k=-L}^{L}\PP_{0}\left[X_{2\Gamma_{-} t^{\sfrac{3}{4}}} =k \right] \leq \frac{C(2L+1)}{\sqrt{2\Gamma_{-} t^{\sfrac{3}{4}}}} \leq t^{\sfrac{-1}{16}},
\end{split}
\end{equation}
if $t$ is large enough.

\par Since we are assuming we are in the event $B^{c} \cap C^{c}$, we bound
\begin{equation}\label{eq:bound_without_pair}
\PP\left[\begin{array}{cl}
\text{a particle that is inside $I$} \\ \text{at time $t$ do not meet} \\ \text{any of its pairs}, B^{c} \cap C^{c}
\end{array}\right] \leq t^{-\frac{1}{16}\frac{t^\frac{1}{4}}{2}} \leq e^{-\frac{1}{32}t^{\frac{1}{4}} \log t}.
\end{equation}

\par The last step is to bound the number of particles inside $H$ at time zero. We choose $\useconstant{c:large_deviation}$ as in~\eqref{eq:large_deviation} and bound
\begin{equation}\label{eq:I_full}
\PP\left[\begin{array}{cl}
\text{there is more than $\useconstant{c:large_deviation}\rho |H|+t$} \\ \text{particles from $\eta$} \\ \text{inside $H$ at time zero}
\end{array}\right] \leq \left[\frac{Z(eR^{-1}(\rho))}{Z(R^{-1}(\rho))}e^{-\useconstant{c:large_deviation}\rho}\right]^{|H|}e^{-t} \leq e^{-t}.
\end{equation}

\par Finally, combining Equations~\eqref{eq:bound_C},~\eqref{eq:bound_B},~\eqref{eq:bound_without_pair} and~\eqref{eq:I_full}, we obtain
\begin{equation}
\begin{split}
\PP[A] & \leq \PP[B]+ \PP[C] + \PP\left[\begin{array}{cl}
\text{there is more than $\useconstant{c:large_deviation}\rho |H|+t$} \\ \text{particles from $\eta$} \\ \text{inside $H$ at time zero}
\end{array}\right] \\
& \quad +(\useconstant{c:large_deviation}\rho |H|+t)\PP\left[\begin{array}{cl}
\text{a particle that is inside $I$} \\ \text{at time $t$ does not meet} \\ \text{any of its pairs}, B^{c} \cap C^{c}
\end{array}\right] \\
& \leq \useconstant{C:coupling_zrp} t (|I|+t)e^{-\useconstant{C:coupling_zrp}^{-1}\epsilon^{2}t^{\sfrac{1}{4}}},
\end{split}
\end{equation}
for some large enough $\useconstant{C:coupling_zrp}$ and all $t$ large. This finishes the proof of Proposition~\ref{prop:coupling}.

\section{The infection process}\label{sec:infection}
~
\par Now that we have constructed the decoupling for the zero range process, we consider the infection process. We first precisely define our model and prove some preliminary results.

\par Given the initial configuration $\eta_{0}$ for the zero range process with density $\rho$, define the set of infected particles $\xi_{0}$ as
\begin{equation}
\xi_{0}(x)=\begin{cases}
\eta_{0}(x), \,\, \text{if } x \leq 0,\\
0, \,\,\,\, \text{if } x>0.
\end{cases}
\end{equation}
Let $\zeta_{0}=\eta_{0}-\xi_{0}$ be the collection of healthy particles.

\par As for the evolution of the process, $\xi+\zeta$ evolves as a zero range process with rate function $g$. Besides, a healthy particle becomes immediately infected when it shares a site with some already infected particle.

\par Observe that this construction satisfies
\begin{equation}
\min\{\xi_{t}(x), \zeta_{t}(x)\}=0\,\, \text{for all } x \in \ZZ \text{ and } t \geq 0.
\end{equation}
This means that, in any non-empty site, either all particles are healthy or all particles are infected.

\par Define the front of the infection wave as
\begin{equation}
r_{t}=\sup\{x:\xi_{t}(x)>0\}.
\end{equation}

\par We now prove some preliminary lemmas regarding the behavior of $r_{t}$. These estimates are uniform over compact sets of positive densities. For the remaining of the section, we fix $0 < \rho_{-} < \rho_{+} < \infty$.
\newconstant{c:bound_A}
\newconstant{c:quadratic_displacement}

\par First, we prove a crude estimate saying that it is unlikely for $r_{t}$ to travel a distance of order $t^{2}$ in time $t$. Let $A(t,t)$ be as in Lemma~\ref{lemma:many_particles} and observe that there exists a positive constant such that $A(t,t) \leq \useconstant{c:bound_A}t$, for $t \geq 1$ and $\rho \in [\rho_{-}, \rho_{+}]$.

\begin{lemma}\label{lemma:infection_1}
There exists a positive constant $\useconstant{c:quadratic_displacement}$ such that
\begin{equation}
\PP_{\rho}\left[\sup_{0 \leq s \leq t}\{r_{s}\}-r_{0} \geq \useconstant{c:bound_A}t^{2}\right] \leq \useconstant{c:quadratic_displacement}e^{-\useconstant{c:quadratic_displacement}^{-1}t}, 
\end{equation}
for all $t \geq 0$ and all $\rho \in [\rho_{-}, \rho_{+}]$.
\end{lemma}

\begin{proof}
By increasing the value of the constant $\useconstant{c:quadratic_displacement}$, we may assume $t \geq 1$.

Write $J=[r_{0}, r_{0}+\useconstant{c:bound_A}t^{2}]$ and observe that, in the event of the statement, either there exists $x \in J$ such that
\begin{equation}
\eta_{s}(x) \geq A(t,t), \,\, \text{for some } s \leq t,
\end{equation}
or this does not happen and, in order for the infection to cross $J$, it must travel through a region that is not dense in particles. This allows us to bound the number of jumps the front of the wave infection can make. Let $X \sim \poisson(\Gamma_{+} t A(t,t))$, we obtain
\begin{equation}
\begin{split}
\PP_{\rho}\left[\sup_{0 \leq s \leq t}\{r_{s}\}-r_{0} \geq \useconstant{c:bound_A}t^{2}\right] & \leq \useconstant{c:bound_A}t^{2}\PP_{\rho}\left[\begin{array}{cl}
\eta_{s}(0) \geq A(t, t), \\ \text{for some } s \in [0,t]
\end{array}\right] +\PP\left[X \geq \useconstant{c:bound_A}t^{2}\right] \\
& \leq \useconstant{c:quadratic_displacement}(t^{2}+1)e^{-\useconstant{c:many_particles}^{-1}t}+e^{-t} \leq \useconstant{c:quadratic_displacement}e^{-\useconstant{c:quadratic_displacement}^{-1}t},
\end{split}
\end{equation}
and the statement follows.
\end{proof}

\par Our next lemma is similar to the last one, but we consider a slightly different event, illustrated in Figure~\ref{fig:lemma_large_displacement_to_the_left_2}.

\begin{figure}[h]\label{fig:lemma_large_displacement_to_the_left_2}
\begin{center}
\begin{tikzpicture}

\filldraw[black!30!, draw=black!30!] (0,0) .. controls (-0.5,0.2) and (-0.7,0.3) .. (-0.9,0.6)--(-3,0.6)--(-3,0);
\filldraw[black!30!, draw=black!30!] (-0.9,0.6) .. controls (-0.9,0.6) and (-1.1,0.9) .. (-1.2,0.9)--(-3,0.9)--(-3,0.6);
\filldraw[black!30!, draw=black!30!] (-1.2,0.9) .. controls (-1.3,0.9) and (-1.5,1.1) .. (-1.6,1.2)--(-3,1.2)--(-3,0.9);
\filldraw[black!30!, draw=black!30!] (-1.6,1.2) .. controls (-1.8,1.4) and (-1.8,1.4) .. (-1.6,1.5)--(-3,1.5)--(-3,1.2);
\filldraw[black!30!, draw=black!30!] (-1.6,1.5) .. controls (-1.4,1.6) and (-1.3,1.8) .. (-1.1,1.9)--(-3,1.9)--(-3,1.5);
\filldraw[black!30!, draw=black!30!] (-1.1,1.9) .. controls (-0.9,2) and (-0.6,2.3) .. (-0.6,2.3)--(-3,2.3)--(-3,1.9);
\filldraw[black!30!, draw=black!30!] (-0.6,2.3) .. controls (-0.5,2.4) and (-0.1,2.4) .. (-0.1, 2.4)--(-3,2.4)--(-3,2.3);
\filldraw[black!30!, draw=black!30!] (-0.1,2.4) .. controls (0.3,2.4) and (0.3,2.7) .. (0.5, 2.7)--(-3,2.7)--(-3,2.4);
\shade[left color=white,right color=black!30!] (-3.1,0) rectangle (-2,2.7);
\fill[white] (-3.1,2.7) rectangle (1,2.8);

\draw[<->](-3.5,0)--(3,0);
\draw[->](0,0)--(0,3);

\draw (0,-0.2)--(0,0);
\draw (-1.75,-0.2)--(-1.75,0);
\node[below] at (-1.75,-0.1) {$\inf_{0 \leq s \leq t}\{r_{s}\}$};\node[below] at (0,-0.1) {$r_{0}$};

\draw[dashed](-1.75,0)--(-1.75,1.4);

\draw (-0.1,2.7)--(0.1,2.7);
\node[right] at (0,2.7) {$t$};

\end{tikzpicture}
\caption{The infimum considered in Lemma~\ref{lemma:infection_3}.}
\end{center}
\end{figure}
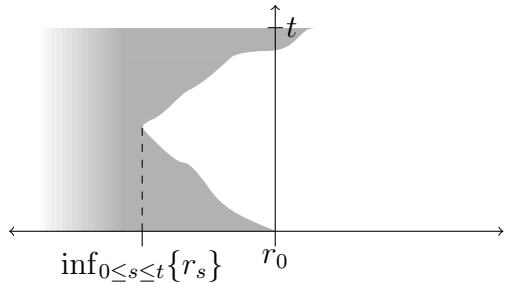

\begin{lemma}\label{lemma:infection_3}
For any $t \geq 0$,
\begin{equation}
\PP_{\rho}\left[r_{0}-\inf_{0 \leq s \leq t}\{r_{s}\}\geq (2\Gamma_{+}+1)t\right]\leq e^{-t}.
\end{equation}
\end{lemma}

\begin{proof}
Simply notice that, on the event above, it is necessary that the first particle on $r_{0}$ jumps more than $(2\Gamma_{+}+1)t$ times before time $t$. This gives the bound
\begin{equation}
\PP_{\rho}\left[r_{0}-\inf_{0 \leq s \leq t}\{r_{s}\}\geq (2\Gamma_{+}+1)t\right] \leq \PP_{\rho}[X \geq (2\Gamma_{+}+1)t] \leq e^{-t},
\end{equation}
where $X \sim \poisson(\Gamma_{+}t)$, and the proof is complete.
\end{proof}

\newconstant{c:large_backtrack}

\par We can also bound the probability that the front of the infection has a big displacement to the right.
\begin{lemma}\label{lemma:infection_2}
There exists a positive constant $\useconstant{c:large_backtrack}$ such that
\begin{equation}
\PP_{\rho}\left[\sup_{s \leq t}\{r_{s}\}-r_{t}\geq (2\Gamma_{+}+1)t\right]\leq \useconstant{c:large_backtrack}e^{-\useconstant{c:large_backtrack}^{-1}t},
\end{equation}
for all $\rho \in [\rho_{-}, \rho_{+}]$.
\end{lemma}

\par Figure~\ref{fig:lemma_large_displacement_to_the_left} helps to illustrate the event in Lemma \ref{lemma:infection_2}.

\begin{figure}[h]\label{fig:lemma_large_displacement_to_the_left}
\begin{center}
\begin{tikzpicture}

\filldraw[black!30!, draw=black!30!] (0,0) .. controls (0.5,0.2) and (0.7,0.3) .. (0.9,0.6)--(-3,0.6)--(-3,0);
\filldraw[black!30!, draw=black!30!] (0.9,0.6) .. controls (0.9,0.6) and (1.1,0.9) .. (1.2,0.9)--(-3,0.9)--(-3,0.6);
\filldraw[black!30!, draw=black!30!] (1.2,0.9) .. controls (1.3,0.9) and (1.5,1.1) .. (1.6,1.2)--(-3,1.2)--(-3,0.9);
\filldraw[black!30!, draw=black!30!] (1.6,1.2) .. controls (1.8,1.4) and (1.8,1.4) .. (1.6,1.5)--(-3,1.5)--(-3,1.2);
\filldraw[black!30!, draw=black!30!] (1.6,1.5) .. controls (1.4,1.6) and (1.3,1.8) .. (1.1,1.9)--(-3,1.9)--(-3,1.5);
\filldraw[black!30!, draw=black!30!] (1.1,1.9) .. controls (0.9,2) and (0.6,2.3) .. (0.6,2.3)--(-3,2.3)--(-3,1.9);
\filldraw[black!30!, draw=black!30!] (0.6,2.3) .. controls (0.5,2.4) and (0.1,2.4) .. (0.1, 2.4)--(-3,2.4)--(-3,2.3);
\filldraw[black!30!, draw=black!30!] (0.1,2.4) .. controls (-0.3,2.4) and (-0.3,2.7) .. (-0.5, 2.7)--(-3,2.7)--(-3,2.4);
\shade[left color=white,right color=black!30!] (-3.1,0) rectangle (-2,2.7);
\fill[white] (-3.1,2.7) rectangle (1,2.8);

\draw[<->](-3,0)--(3,0);
\draw[->](0,0)--(0,3);

\draw (-0.5,-0.2)--(-0.5,0);
\draw (1.75,-0.2)--(1.75,0);
\node[below] at (1.75,-0.1) {$\sup_{0 \leq s \leq t}\{r_{s}\}$};
\node[below] at (-0.5,-0.1) {$r_{t}$};

\draw[dashed](-0.5,0)--(-0.5,2.7);
\draw[dashed](1.75,0)--(1.75,1.4);

\draw (-0.1,2.7)--(0.1,2.7);
\node[right] at (0,2.7) {$t$};

\end{tikzpicture}
\caption{The supremum in the event considered in Lemma~\ref{lemma:infection_2}.}
\end{center}
\end{figure}
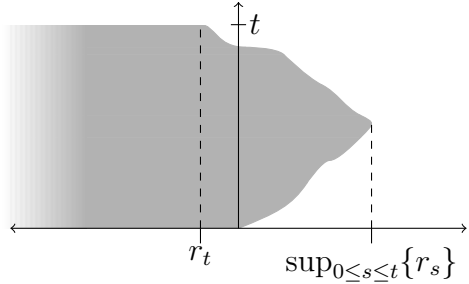

\begin{proof}
Let $B$ denote the event in the statement of the lemma, write $I=[-(2\Gamma_{+}+2)t, \useconstant{c:bound_A}t^{2}]$ and notice that
\begin{equation}\label{eq:bound_1}
\begin{split}
\PP_{\rho}[B] & \leq \PP_{\rho}\left[r_{0} \notin [-t,0]\right]+ \PP_{\rho}\left[\inf_{s \leq t}r_{s} \leq -(2\Gamma_{+}+2)t, r_{0} \geq -t\right] \\
& \qquad + \PP_{\rho}\left[ \sup_{s \leq t} r_{s} \geq \useconstant{c:bound_A}t^{2} \right]+\PP_{\rho}\left[B, r_{s} \in I, \text{ for all } s \leq t\right].
\end{split}
\end{equation}

Combining Lemmas~\ref{lemma:infection_1} and~\ref{lemma:infection_3}, we easily obtain that
\begin{equation}
\PP_{\rho}[B] \leq ce^{-c^{-1}t}+\PP_{\rho}\left[B, r_{s} \in I, \text{ for all } s \leq t\right].
\end{equation}

To bound the last probability of the last event above, observe that, if it holds, then either there exists some particle from outside $H=[-(5\Gamma_{+}+2)t, \useconstant{c:bound_A}t^{2}+3\Gamma_{+}t]$ enters the interval $I$ before time $t$, or some particle that starts inside $H$ jumps many times before time $t$. Using the same strategy as in Lemma~\ref{lemma:many_particles}, concentration of the number of particles inside $H$ and the fact that each particle jumps at most $\poisson(\Gamma_{+}t)$ times before time $t$ we obtain
\begin{equation}
\begin{split}
\PP_{\rho}\left[B, r_{s} \in I, \text{ for all } s \leq t\right] & \leq \PP_{\rho}\left[\begin{array}{cl}
\text{some particle that starts outside} \\ I \text{ enters } H \text{ before time } t 
\end{array}\right] \\
& \qquad + \PP_{\rho}\left[\begin{array}{cl}
\text{some particle inside } I \text{ jumps more} \\ \text{than } (2\Gamma_{+}+1)t \text{ times before time } t 
\end{array}\right] \\
& \leq c(t^{2}+t+1)e^{-c^{-1}t}.
\end{split}
\end{equation}
Combining the last expression above with~\eqref{eq:bound_1} completes the proof.
\end{proof}

\par To finish this section, we introduce the space-time translated infection process. Fix $m =(x,t) \in \ZZ \times [0, \infty)$ and define the collection of infected particles as
\begin{equation}
\xi_{0}^{m}(y)=\begin{cases}
\eta_{t}(y), \,\, \text{if } y \leq x,\\
0, \,\,\,\, \text{if } y>x.
\end{cases}
\end{equation}
As before, $\zeta_{0}^{m}=\eta_{t}-\xi_{0}^{m}$ denotes the collection of healthy particles. The evolution of the infection is the same, and the front of the infection wave is
\begin{equation}
r_{s}(m)=\sup\{y \in \ZZ: \xi_{s}^{m}(y)>0\}, \qquad s \geq t.
\end{equation}

\section{Finite velocity}\label{sec:finite_vel}
~
\par We now begin a more in depth study of our infection process. This section aims to prove Theorem~\ref{teo:finite_velocity}. We split the discussion in three subsections. The first subsection contains some notation we will need to develop our multiscale renormalisation, which can be found in Subsection~\ref{subsec:renorm}. Subsection~\ref{subsec:finite_vel} contains the proof of Theorem~\ref{teo:finite_velocity}.


\subsection{The box notation}\label{subsec:box_notation}
~
\par We begin by introducing the sequence of scales $(L_{k})_{k \in \NN_{0}}$ as
\begin{equation}\label{eq:scales}
L_{0}=100 \qquad \text{and} \qquad L_{k+1}=L_{k}^{3}.
\end{equation}
We will also write $\ell_{k}=\lfloor L_{k}^{\sfrac{1}{2}} \rfloor$.

\par For $k \in \NN_{0}$, define the box
\begin{equation}\label{eq:renormalisation_box}
B_{k}=[-\ell_{k} L_{k}^{2},\ell_{k} L_{k}^{2}] \times [0,L_{k}],
\end{equation}
and, for $m \in \ZZ \times L_{k} \NN_{0}$, let $B_{k}(m)$ denote the translated box $B_{k}(m)=m+B_{k}$.

\par Define also the sequence of velocities
\begin{equation}\label{eq:velocities}
v_{0}=v>0 \qquad \text{and} \qquad v_{k+1}=v_{k}+\frac{1}{(k+1)^{2}},
\end{equation}
where $v$ is a positive value that will be chosen afterwards to be sufficiently large.

\par We want to bound the probability of the events where $r_{t}$ travels fast to the right. However, the continuous time nature of the process implies that events of this form do not have a bounded support. Therefore, we will introduce a well chosen event that treats the possibility that either $r_{t}$ leaves the box $B_{k}$ before time $L_{k}$ or it is far to the right at time $L_{k}$. For $k \in \NN_{0}$, define the set
\begin{equation}\label{eq:boundary_subset}
R_{k} =\Big(\{ \ell_{k} L_{k}^{2}\} \times [0, L_{k}] \Big) \cup \Big( [v_{k}L_{k}, \ell_{k} L_{k}^{2}] \times \{L_{k}\} \Big).
\end{equation}
Figure~\ref{fig:box} contains a representation of $B_{k}$ and $R_{k}$. For $m \in \ZZ \times L_{k} \NN_{0}$, define $R_{k}(m)=m+R_{k}$.

\begin{figure}\label{fig:box}
\begin{center}
\begin{tikzpicture}

\draw[<->] (-4,0)--(4,0);
\draw[->](0,0)--(0,3);

\draw[thick](-3,0) rectangle (3,2);

\draw[->](0,0)--(3,3);
\node[right] at (3,3){$x=v_{k}t$};

\draw[line width=3pt](2,2)--(3,2)--(3,0);
\node[right] at (3,1){$R_{k}$};

\draw (0,0) to [curve through={(0,0) .. (-0.3, 0.1) .. (0, 0.6) .. (1.2, 0.8) .. (1.5, 1.1)}] (3,1.7);

\end{tikzpicture}
\caption{The box $B_{k}$, the set $R_{k}$ and the event $E_{k}$.}
\end{center}
\end{figure}
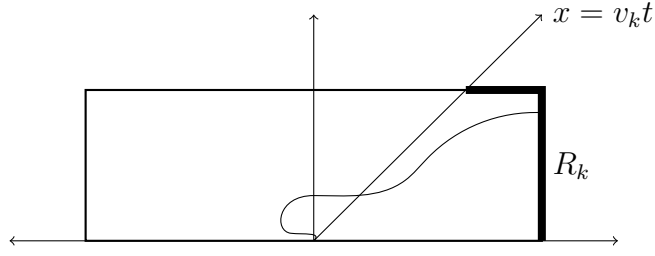

\par The event we consider is defined as follows. For $m =(x, sL_{k}) \in \ZZ \times L_{k} \NN_{0}$, consider
\begin{equation}\label{eq:renorm_event}
E_{k}(m)=\left\{\begin{array}{ll}
\text{$r_{0}(m)=x$ and $(r_{t}(m))_{t>0}$ first touches} \\ \text{the boundary of $B_{k}(m)$ in $R_{k}(m)$}
\end{array}\right\}.
\end{equation}
See Figure~\ref{fig:box} for a representation of the event $E_{k}$. Observe that the events $E_{k}(m)$ are non-decreasing and have support in $B_{k}(m)$. When $m=(0,0)$, we will omit it and denote $E_{k}(0,0)$ simply by $E_{k}$.

\par We introduce the sequence of densities. Fix $\rho_{0} >0$ and define
\begin{equation}\label{eq:density_scales}
\rho_{k}=\rho_{k+1}(1+L_{k}^{-\sfrac{1}{16}}).
\end{equation}
The sequence $(\rho_{k})_{k \in \NN_{0}}$ is decreasing and $\rho_{\infty}=\lim \rho_{k}$ is positive.

\par Define, for $m \in \ZZ \times L_{k}\NN_{0}$, the probability of the bad events as
\begin{equation}\label{eq:p_k}
p_{k}=\PP_{\rho_{k}}[E_{k}(m)].
\end{equation}
By translation invariance, the probability above does not depend on the value of~$m$.

\begin{remark}\label{remark:velocity_dependence}
Even though $p_{k}$ also depends on the value of $v_{k}$ which is determined by the fixed value of $v_{0}=v$, we omit these dependencies.
\end{remark}

\newconstant{c:bad_support}

\par We also introduce the event
\begin{equation}\label{eq:infection_travels_far}
D_{k}(m)=\{r_{t}(m) \in B_{k}(m), \text{ for all } t \in [0,L_{k}]\}.
\end{equation}
Lemmas~\ref{lemma:infection_1} and~\ref{lemma:infection_3} imply that, given $\rho_{-}<\rho_{+}$, there exists $\useconstant{c:bad_support}>0$ such that
\begin{equation}\label{eq:probability_bad_support}
\PP_{\rho}[D_{k}^{c}] \leq \useconstant{c:bad_support}e^{-\useconstant{c:bad_support}^{-1}L_{k}},
\end{equation}
for all $k \in \NN_{0}$ and $\rho \in [\rho_{-}, \rho_{+}]$.

\newconstant{c:bound_M}

\par Finally, let $M_{k}$ denote the set of values $m$ for which the translated box $B_{k}(m)$ still intersects the larger box $B_{k+1}$, more precisely,
\begin{equation}\label{eq:M_k}
M_{k}=\{m \in \ZZ \times L_{k} \NN_{0}: B_{k}(m) \cap B_{k+1} \neq \emptyset\},
\end{equation}
and observe that
\begin{equation}
|M_{k}| \leq \useconstant{c:bound_M}L_{k+1}^{4}.
\end{equation}

\subsection{Estimates on $p_{k}$}\label{subsec:renorm}
~
\par Our next step is to prove that $p_{k}$ decreases very fast when $v_{0}$ is chosen large enough. This is done in three lemmas, proved in this subsection.

\newconstant{c:renorm}

\par The first lemma we prove is a recursive inequality that relates $p_{k}$ to $p_{k+1}$.
\begin{lemma}\label{lemma:renorm_1}
There exists $k_{0}$ such that, for all choice of $v_{0}$ and $k \geq k_{0}$,
\begin{equation}
p_{k+1}\leq \useconstant{c:renorm}L_{k+1}^{28}[p_{k}^{4}+e^{-\useconstant{c:renorm}^{-1}\rho_{\infty}^{2}L_{k}^{\sfrac{1}{8}}}].
\end{equation}
\end{lemma}

\begin{proof}
Fix $k_{0} \in \NN_{0}$ such that, for all $k \geq k_{0}$,
\begin{equation}
\frac{1}{6(k+1)^{2}} > \frac{1}{L_{k}^{\sfrac{1}{2}}}.
\end{equation}

We will prove that the sequence of events $(E_{k})_{k \geq 0}$ is cascading: the occurrence of $E_{k+1}$ implies that many $E_{k}(m)$ hold. Fix $k \geq k_{0}$ and assume we are in the event $E_{k+1}\cap D_{k+1}$. We claim that
\begin{equation*}
\begin{array}{cl}
\text{either } D_{k}(m)^{c} \text{ holds for some } m \in M_{k} \text{ or there are} \\
\text{seven elements } m_{i}=(x_{i},s_{i})\in M_{k}, 1 \leq i \leq 7, \text{ with}\\
s_{i} \neq s_{j}, \text{ if } i \neq j, \text{ such that } E_{k}(m) \text{ holds}.
\end{array}
\end{equation*}

The proof follows by contradiction. Assume we are in the event $E_{k+1} \cap D_{k+1}$, that $D_{k}(m)$ holds for all $m \in M_{k}$, and that $E_{k}(m)$ holds for at most six values of $m \in M_{k}$ with different time coordinates.

Observe that, if $E_{k}(m) \cap D_{k}(m)$ holds, $r_{t}(m)$ has a  maximum displacement of $\ell_{k} L_{k}^{2}$ before time $L_{k}$. Thus, we have
\begin{equation}
\begin{split}
r_{L_{k+1}}-r_{0} & =\sum_{j=0}^{L_{k}^{2}-1}r_{L_{k}}(r_{jL_{k}},jL_{k})-r_{jL_{k}} \\
& \leq 6L_{k}^{5/2}+(L_{k}^{2}-6)v_{k}L_{k} \\
& \leq 6L_{k+1}\left(\frac{1}{L_{k}^{\sfrac{1}{2}}}-\frac{1}{6(k+1)^{2}}\right)+L_{k+1}v_{k+1} \\
& < L_{k+1}v_{k+1}.
\end{split}
\end{equation}
This implies that we are in $E_{k+1}^{c} \cup D_{k+1}^{c}$, a contradiction.

Thus, on the event $E_{k+1} \cap D_{k+1}$, either some $D_{k}(m)^{c}$ with $m \in M_{k}$ occurs, or there are seven elements $m_{i}=(x_{i},s_{i})\in M_{k}$, $1 \leq i \leq 7$, with $s_{i} \neq s_{j}$, if $i \neq j$, such that $E_{k}(m_{i})$ occurs. 

Assume we are in the last case described above. We will use a union bound over all choices of $m_{i} \in M_{k}$. Fix one such choice and observe that $L_{k} \leq s_{i+2}-s_{i} \leq L_{k+1}$, for $1 \leq i \leq 5$.

We now apply Theorem~\ref{teo:vertical_decoupling} considering the event $E_{k}(m_{1})$ and the intersection $\cap_{3 \leq i \leq 7}E_{k}(m_{i})$, as in Figure~\ref{fig:renormalisation}: we can use boxes of side length $5\ell_{k+1} L_{k+1}$. Set $\epsilon=\frac{1}{3}(\rho_{k}-\rho_{k+1})=\rho_{k+1}\frac{L_{k}^{-\sfrac{1}{16}}}{3}$ and estimate
\begin{equation}
\begin{split}
\PP_{\rho_{k+1}}\left[\bigcap_{i=1}^{7} E_{k}(m_{i})\right] & \leq \PP_{\rho_{k+1}+\epsilon}\left[E_{k}(m_{1})\right]\PP_{\rho_{k+1}+\epsilon}\left[\bigcap_{i=3}^{7} E_{k}(m_{i})\right]+\useconstant{c:vertical_decoupling} L_{k+1}^{3}e^{-\useconstant{c:vertical_decoupling}^{-1}\rho_{\infty}^{2}L_{k}^{\sfrac{1}{8}}}\\
& \leq \PP_{\rho_{k}}\left[E_{k}(m_{1})\right]\PP_{\rho_{k+1}+\epsilon}\left[\bigcap_{i=3}^{7} E_{k}(m_{i})\right]+\useconstant{c:vertical_decoupling} L_{k+1}^{3}e^{-\useconstant{c:vertical_decoupling}^{-1}\rho_{\infty}^{2}L_{k}^{\sfrac{1}{8}}}.
\end{split}
\end{equation}

We apply Theorem~\ref{teo:vertical_decoupling} two more times: in the first use, we consider the events $E_{k}(m_{3})$ and $\bigcap_{i=5}^{7} E_{k}(m_{i})$. The last time uses the events $E_{k}(m_{5})$ and $E_{k}(m_{7})$. These computations yield the bound
\begin{equation}
\PP_{\rho_{k+1}}\left[\bigcap_{i=1}^{7} E_{k}(m_{i})\right] \leq \PP_{\rho_{k}}[E_{k}]^{4} +3\useconstant{c:vertical_decoupling} L_{k+1}^{3}e^{-\useconstant{c:vertical_decoupling}^{-1}\rho_{\infty}^{2}L_{k}^{\sfrac{1}{8}}}.
\end{equation}

\begin{figure}\label{fig:renormalisation}
\begin{center}
\begin{tikzpicture}

\draw (-4,0) rectangle (4,3);

\draw (-1,0) rectangle (1,0.3);
\draw (-0.5,0.3) rectangle (1.5,0.6);
\draw (-0.5,0.9) rectangle (1.5,1.2);
\draw (0,1.2) rectangle (2,1.5);
\draw (0,2.1) rectangle (2,2.4);
\draw (1,2.4) rectangle (3,2.7);
\draw (1.5,2.7) rectangle (3.5,3);

\draw[thick] (0,0)--(0.5,0.3)--(0.9, 0.6)--(0.5,0.9)--(0.9,1.2)--(1.3,1.5)--(1,2.1)--(2,2.4)--(2.5,2.7)--(3.3,3);

\draw[ultra thick](-1,0) rectangle (1,0.3);]
\draw[ultra thick](-0.5,0.9) rectangle (3.5,3);

\end{tikzpicture}
\caption{The boxes in the cascading event and the supports of the functions in the first application of Theorem~\ref{teo:vertical_decoupling}.}
\end{center}
\end{figure}
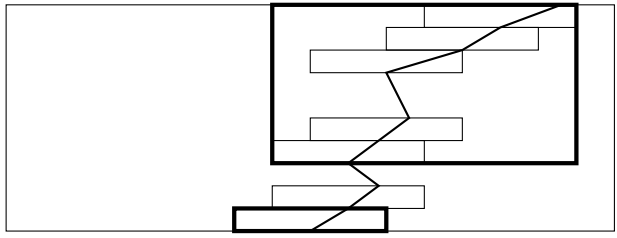

By changing constants, it is easy to conclude that
\begin{equation}
\begin{split}
p_{k+1} & \leq \PP_{\rho_{k+1}}[E_{k+1} \cap D_{k+1}]+\PP_{\rho_{k+1}}[D_{k+1}^{c}] \\
&\leq |M_{k}|^{7}(\PP_{\rho_{n}}[E_{k}]^{4} +3\useconstant{c:vertical_decoupling} L_{k+1}^{3}e^{-\useconstant{c:vertical_decoupling}^{-1}\rho_{\infty}^{2}L_{k}^{\sfrac{1}{8}}}) \\
& \qquad +|M_{k}|\PP_{\rho_{k+1}}[D_{k}^{c}]+\PP_{\rho_{k+1}}[D_{k+1}^{c}] \\
& \leq \useconstant{c:renorm}L_{k+1}^{28}[p_{k}^{4}+e^{-\useconstant{c:renorm}^{-1}\rho_{\infty}^{2}L_{k}^{\sfrac{1}{8}}}],
\end{split}
\end{equation}
and the statement follows.
\end{proof}

\par Now we prove a recursive estimate on $p_{k}$.
\begin{lemma}\label{lemma:renorm_2}
There exists $k_{1} \geq k_{0}$ such that, for $k \geq k_{1}$ and any choice of $v_{0}$, if
\begin{equation}\label{eq:recursive_relation}
p_{k} \leq e^{-\log^{\sfrac{5}{4}}L_{k}},
\end{equation}
then
\begin{equation}
p_{k+1} \leq e^{-\log^{\sfrac{5}{4}}L_{k+1}}.
\end{equation}
\end{lemma}

\begin{proof}
Observe that $3^{\sfrac{5}{4}} \leq 4$. Assume that~\eqref{eq:recursive_relation} holds for some $k \geq k_{0}$. Recall that $L_{k+1}=L_{k}^{3}$ and use Lemma~\ref{lemma:renorm_1} to conclude that
\begin{equation}
\begin{split}
e^{\log^{\sfrac{5}{4}}L_{k+1}}p_{k+1} & \leq \useconstant{c:renorm}L_{k+1}^{28}[p_{k}^{4}+e^{-\useconstant{c:renorm}^{-1}\rho_{\infty}^{2}L_{k}^{\sfrac{1}{8}}}]e^{\log^{\sfrac{5}{4}}L_{k+1}} \\
& \leq \useconstant{c:renorm}L_{k+1}^{28}[e^{-4\log^{\sfrac{5}{4}}L_{k}}+e^{-\useconstant{c:renorm}^{-1}\rho_{\infty}^{2}L_{k}^{\sfrac{1}{8}}}]e^{3^{\sfrac{5}{4}}\log^{\sfrac{5}{4}}L_{k}}\\
& \leq \useconstant{c:renorm}L_{k+1}^{28}[e^{(-4+3^{\sfrac{5}{4}})\log^{\sfrac{5}{4}}L_{k}}+e^{-\useconstant{c:renorm}^{-1}\rho_{\infty}^{2}L_{k}^{\sfrac{1}{8}}+3^{\sfrac{5}{4}}\log^{\sfrac{5}{4}}L_{k}}].
\end{split}
\end{equation}

Now simply choose $k_{1} \geq k_{0}$ such that, if $k \geq k_{1}$, then
\begin{equation*}
\useconstant{c:renorm}L_{k+1}^{28}[e^{(-4+3^{\sfrac{5}{4}})\log^{\sfrac{5}{4}}L_{k}}+e^{-\useconstant{c:renorm}^{-1}\rho_{\infty}^{2}L_{k}^{\sfrac{1}{8}}+3^{\sfrac{5}{4}}\log^{\sfrac{5}{4}}L_{k}}]<1.
\end{equation*}
This concludes the proof.
\end{proof}

\par The last step is to verify that, if $v_{0}$ is chosen large enough,~\eqref{eq:recursive_relation} holds for some $k \geq k_{1}$.
\begin{lemma}\label{lemma:renorm_3}
There exist $v_{0}$ and $k_{2} \geq k_{1}$ such that $p_{k_{2}} \leq e^{-\log^{\sfrac{5}{4}}L_{k_{2}}}$.
\end{lemma}

\begin{proof}
Our strategy is to choose one value of $v_{0}^{k}$ for each $k$ at first. We then fix $k_{2}$ large enough and choose the corresponding $v_{0}^{k}$.

For $k \geq k_{1}$, set $v_{0}^{k}$ in such a way that ${v}_{k}^{k}=\ell_{k} L_{k}$, and observe that for this velocity,
\begin{equation}
E_{k} \subset \left\{\sup_{s \leq L_{k}}\{r_{s}\}-r_{0} \geq \useconstant{c:bound_A} L_{k}^{2}\right\}.
\end{equation}

Now, Lemma~\ref{lemma:infection_1} implies that
\begin{equation}
p_{k}(v_{0}^{k}) \leq \PP_{\rho_{k}}\left[\sup_{s \leq L_{k}}\{r_{s}\}-r_{0} \geq \useconstant{c:bound_A} L_{k}^{2}\right] \leq \useconstant{c:quadratic_displacement}e^{-\useconstant{c:quadratic_displacement}^{-1}L_{k}}.
\end{equation}

Increasing the value of $k$ if necessary gives the desired bound. Once we have a value for $k$, we can choose the corresponding value for $v_{0}^{k}$.
\end{proof}

\subsection{Proof of Theorem~\ref{teo:finite_velocity}}\label{subsec:finite_vel}
~
\par In this section, we conclude the proof of Theorem~\ref{teo:finite_velocity}. We use the multiscale renormalisation scheme developed in the last subsection to prove that~$r_{t}$ has finite velocity.

\par Define the space-time half-plane
\begin{equation}\label{eq:half_space}
\mathscr{H}_{v,L}=\{(x,t) \in \ZZ \times \RR_{+}: x \geq tv+L\}.
\end{equation}
We will prove that the probability that $r_{t} \in \mathscr{H}_{v,L}$, for some $t \geq 0$, decays fast with $L$ when $v$ is large enough. We already have information about $r_{t}$ for the times $L_{k}$. All that is necessary now is to interpolate between these times.

\par Fix $v_{0}$ and $k_{2}$ as in Lemma~\ref{lemma:renorm_3} and define
\begin{equation}\label{eq:terminal_velocity}
\bar{v}=v_{\infty}=\lim_{k \to \infty} v_{k}.
\end{equation}
Define the events $\bar{E}_{k}(m)$ as in~\eqref{eq:renorm_event} but with $v_{k}$ replaced by $\bar{v}$. Observe that we have
\begin{equation}\label{eq:renorm_bound}
\PP_{\rho_{\infty}}[\bar{E}_{k}(m)] \leq \PP_{\rho_{k}}[E_{k}(m)] \leq e^{-\log^{\sfrac{5}{4}}L_{k}}, \qquad \text{for all } k \geq k_{3}.
\end{equation}

\begin{proof}[Proof of Theorem~\ref{teo:finite_velocity}]
Notice that, if
\begin{equation}\label{eq:conditioning_step}
\PP_{\rho_{\infty}}\left[\begin{array}{cl}
r_{t} \in \mathscr{H}_{\bar{v},L},
\text{ for some } t \geq 0, \\ \text{and } r_{0}=0
\end{array}\right] \leq \useconstant{c:finite_speed}e^{-\useconstant{c:finite_speed}^{-1}\log^{\sfrac{5}{4}}L},
\end{equation}
then
\begin{equation}
\begin{split}
\PP_{\rho_{\infty}}\left[\begin{array}{cl}
r_{t} \in \mathscr{H}_{\bar{v},L}, \\
\text{for some } t \geq 0
\end{array}\right] & \leq \sum_{y=0}^{\infty}\PP_{\rho_{\infty}}\left[\begin{array}{cl}
r_{t} \in \mathscr{H}_{\bar{v},L},
\text{ for some } t \geq 0, \\ \text{and } r_{0}=-y
\end{array}\right] \\
& \leq \sum_{y=0}^{\infty}\useconstant{c:finite_speed}e^{-\useconstant{c:finite_speed}^{-1}\log^{\sfrac{5}{4}}(L+y)} \leq \useconstant{c:finite_speed}e^{-\useconstant{c:finite_speed}^{-1}\log^{\sfrac{5}{4}}L}.
\end{split}
\end{equation}
Hence, we may condition on the event $\{r_{0}=0\}$.

By changing constants, we may assume that $L \geq L_{k_{2}}$. Choose $\tilde{k} \geq k_{2}$ such that
\begin{equation}
L_{\tilde{k}}\leq L < L_{\tilde{k}+1}.
\end{equation}

For $m =(x,s) \in \ZZ \times L_{k}\NN_{0}$, we define the event where $r(m)$ does not travel very far in time $L_{k}$, more precisely,
\begin{equation}
H_{k}(m)=\left\{\begin{array}{cl}
\sup_{0 \leq t \leq L_{k}}r_{t}(m)-x \leq (\bar{v}+1)L_{k} \\
\text{and} \\
x-\inf_{0 \leq t \leq L_{k}}r_{t}(m) \leq 2(\Gamma_{+}+1)L_{k}
\end{array}\right\},
\end{equation}
and observe that Lemmas~\ref{lemma:infection_2} and~\ref{lemma:infection_3} imply, by possibly changing the value of the constant,
\begin{equation}\label{eq:probability_E}
\PP_{\rho_{\infty}}[\bar{E}_{k}(m)^{c} \cap H_{k}(m)^{c}] \leq \useconstant{c:large_backtrack}e^{-\useconstant{c:large_backtrack}^{-1}L_{k}}.
\end{equation}

We will define an event where $r_{t}$ is well-behaved. Recall~\eqref{eq:M_k} and consider
\begin{equation}
\tilde{B}_{\tilde{k}}=\bigcap_{k \geq \tilde{k}}\bigcap_{m \in M_{k}} \bar{E}_{k}(m)^{c} \cap H_{k}(m).
\end{equation}
In the event above, we have bounds for $r_{t}$ at the times $L_{k}$ and we also know that the front does not travel far away during the time intervals of length $L_{k}$.

\newconstant{c:bad_event_bound}

Observe that Equations~\eqref{eq:renorm_bound} and~\eqref{eq:probability_E} imply that
\begin{equation}\label{eq:prob_tilde_B_complement}
\begin{split}
\PP_{\rho_{\infty}}[\tilde{B}_{\tilde{k}}^{c}] & \leq \sum_{k \geq \tilde{k}}\sum_{m \in M_{k}}\PP_{\rho_{\infty}}[\bar{E}_{k}(m)]+\PP_{\rho_{\infty}}[\bar{E}_{k}(m)^{c} \cap H_{k}(m)^{c}]\\
& \leq \useconstant{c:bad_event_bound}\sum_{k \geq \tilde{k}}L_{k}^{12}e^{-\log^{\sfrac{5}{4}}L_{k}} \leq \useconstant{c:bad_event_bound}L_{\tilde{k}}^{13}e^{-\log^{\sfrac{5}{4}}L_{\tilde{k}}} \\
& \leq \useconstant{c:bad_event_bound}L^{13}e^{-\useconstant{c:bad_event_bound}^{-1}\log^{\sfrac{5}{4}}L},
\end{split}
\end{equation}
where the tail bound in the second line above is proved in an analogous way as Lemma D.1 of~\cite{rwrw}.

We now study the event $\tilde{B}_{\tilde{k}}$. Consider
\begin{equation}\label{eq:set_times}
J_{\tilde{k}}=\bigcup_{k \geq \tilde{k}}\bigcup_{\ell=0}^{\sfrac{L_{k+1}}{L_{k}}}\{\ell L_{k}\}.
\end{equation}

We claim that, on $\tilde{B}_{\tilde{k}}\cap \{r_{0}=0\}$,
\begin{equation}\label{eq:bounded_velocity}
r_{t} \leq \bar{v}t, \qquad \text{for all } t \in J_{\tilde{k}}.
\end{equation}
To see why this is true, fix $k \geq \tilde{k}$ and use induction on $\ell$. The claim is clearly true for $\ell=0$. Suppose it is true for some $\ell <\sfrac{L_{k+1}}{L_{k}}$. Observe that, since we are in $H_{k}(m)$, for $m \in M_{k}$, $(r_{\ell L_{k}}, \ell L_{k})$ belongs to $B_{k+1}$. Using that $E_{k}(r_{\ell L_{k}})^{c}$ holds, we have
\begin{equation}
\begin{split}
r_{(\ell+1)L_{k}}& = \left(r_{L_{k}}(r_{\ell L_{k}},\ell L_{k})-r_{\ell L_{k}}\right)+r_{\ell L_{k}} \\
& \leq \bar{v}L_{k}+\bar{v}\ell L_{k} = \bar{v}(\ell+1)L_{k}.
\end{split}
\end{equation}

It remains to interpolate the relation in~\eqref{eq:bounded_velocity} for positive values of $t$. Consider initially $t \geq L$. Let $\kappa$ be the smallest $k \geq \tilde{k}$ such that
\begin{equation}
\ell L_{\kappa} \leq t < (\ell+1)L_{\kappa}, \qquad \text{for some } \ell < \frac{L_{\kappa+1}}{L_{\kappa}}.
\end{equation}
Let $\bar{\ell}$ denote the unique value of $\ell$ and observe that $\bar{\ell} \geq 1$.

We compute
\begin{equation}
\begin{split}
r_{t}& = \left(r_{t-\bar{\ell}L_{\kappa}}(r_{\bar{\ell}L_{\kappa}},\bar{\ell}L_{\kappa})-r_{\bar{\ell}L_{\kappa}}\right)+r_{\bar{\ell}L_{\kappa}} \\
& \leq (\bar{v}+1)L_{\kappa}+\bar{v}\bar{\ell}L_{\kappa} \leq (2\bar{v}+1)t.
\end{split}
\end{equation}

We now consider $t \leq L$. Observe that, on $\tilde{B}_{\tilde{k}} \cap \{r_{0}=0\}$, we have $r_{L} \leq (2\bar{v}+1)L$. Now, Lemma~\ref{lemma:infection_2} implies
\begin{equation}
\begin{split}
\PP_{\rho_{\infty}} & \left[\sup_{s \leq L}r_{s} \geq 2(\bar{v}+\Gamma_{+}+1)L, \tilde{B}_{\tilde{k}}\cap \{r_{0}=0\}\right]\\
& \qquad \qquad \qquad \leq \PP_{\rho_{\infty}}\left[r_{L}-\sup_{s \leq L}r_{s} \geq (2\Gamma_{+}+1)L\right] \leq \useconstant{c:large_backtrack}e^{-\useconstant{c:large_backtrack}L}.
\end{split}
\end{equation}

Combining the last expression above with~\eqref{eq:prob_tilde_B_complement}, we obtain
\begin{equation}
\PP_{\rho_{\infty}}\left[\begin{array}{cl}
r_{t} \in \mathscr{H}_{2\bar{v}+1,2(\bar{v}+\Gamma_{+}+1)L}, \\
\text{for some } t \geq 0, \text{ and } r_{0}=0
\end{array}\right] \leq \useconstant{c:positive_speed}e^{-\useconstant{c:positive_speed}^{-1}\log^{\sfrac{5}{4}}L}.
\end{equation}

By changing constants, the proof is complete.
\end{proof}

\section{Positive velocity}\label{sec:positive_vel}

\par This section contains the proof of Theorem~\ref{teo:positive_velocity}. This proof is also based on multiscale renormalisation. One may try to consider a similar event as the one in the renormalisation used in the proof of Theorem~\ref{teo:finite_velocity}, and prove that the event where the front of the wave does not travel to the right has small positive probability. However, when proving an analogous of Lemma~\ref{lemma:renorm_3}, it is necessary to understand more refined properties of the process in order to prove that, at a fixed large time, the wave front indeed travels to the right with large probability. We choose to consider a slightly different approach, proving that, in a positive proportion of time, the front of the infection wave has more than one particle, producing a drift to the right. A similar approach was also used by~\cite{bht} and~\cite{ks}.

\subsection{Simultaneous decoupling}
~
\par For the proof of positive velocity, it is not possible to apply the decoupling stated in Therorem~\ref{teo:vertical_decoupling} for the class of events we consider in the renormalisation. In this subsection we provide a stronger version of the decoupling.

\par For $\rho < \rho'$, we construct the measure $\PP_{\rho, \rho'}$ in the following way. Begin with two initial configurations that satisfy $\eta_{0}(x)\leq \eta'_{0}(x)$ (this can be done using the usual monotone coupling) and use one copy of the graphical construction presented in Subsection~\ref{subsec:graphical} to evolve both processes at the same time.  Whenever a particle jumps, it goes on top of its respective pile and particles of $\eta$ are also seen as particles of the process $\eta'$.

\par The probability measure $\PP_{\rho, \rho'}$ provides the construction of two zero range processes, $\eta$ and $\eta'$, with respective densities $\rho$ and $\rho'$ and that satisfy $\eta_{t}(x) \leq \eta'_{t}(x)$, for all $(x,t) \in \ZZ \times \RR_{+}$.

\newconstant{c:simultaneous_decoupling}
\newbigconstant{C:simultaneous_decoupling}

\par We prove a decoupling for the collection of measures $\PP_{\rho, \rho'}$.
\begin{prop}\label{prop:simultaneous_decoupling}
Fix $0< \rho_{-} < \rho_{+}$. There exist positive constants $\useconstant{c:simultaneous_decoupling}=\useconstant{c:simultaneous_decoupling}(\rho_{-}, \rho_{+})$ and $\usebigconstant{C:simultaneous_decoupling}=\usebigconstant{C:simultaneous_decoupling}(\rho_{-}, \rho_{+})$ such that, for any two boxes $B_{1}$ and $B_{2}$ with side-length $s>0$ that satisfy
\begin{equation}
\dist_{V}=\dist_{V}(B_{1}, B_{2}) \geq \usebigconstant{C:simultaneous_decoupling},
\end{equation}
and any two functions $f_{1}(\eta, \eta')$ and $f_{2}(\eta, \eta')$ satisfying
\begin{enumerate}
\item $f_{i}$ is supported in $B_{i}$;
\item $0 \leq f_{i}(\eta, \eta') \leq 1$ almost surely;
\item $f_{i}$ is non-increasing in $\eta$ and non-decreasing in $\eta'$;
\end{enumerate}
we have the following. For any $\rho \leq \rho'$ and $\epsilon \in (0,1]$ such that $\rho_{-} \leq \rho-\epsilon \leq \rho' \leq \rho_{+}$,
\begin{equation}\label{eq:decoupling_estimate_2}
\EE_{\rho, \rho'}[f_{1}f_{2}] \leq \EE_{\rho-\epsilon, \rho'+\epsilon}[f_{1}]\EE_{\rho-\epsilon, \rho'+\epsilon}[f_{2}]+\useconstant{c:simultaneous_decoupling} \dist_{V}(\dist_{V}+s+1)e^{-\useconstant{c:simultaneous_decoupling}^{-1}\epsilon^{2}\dist_{V}^{\sfrac{1}{4}}}.
\end{equation}
\end{prop}

\newbigconstant{c:simultaneous_coupling_2}
\newconstant{c:simultaneous_coupling}

\par The proof follows exactly the same steps of the proof of Theorem~\ref{teo:vertical_decoupling}. The existence of a coupling with the same characteristics of the one in Proposition~\ref{prop:coupling} is guaranteed by the next result.
\begin{prop}\label{prop:simultaneous_coupling}
Fix $0 < \rho_{-} < \rho_{+}$. There exist positive constants $\useconstant{c:simultaneous_coupling}$ and $\usebigconstant{c:simultaneous_coupling_2}$ such that, for any $t \geq \usebigconstant{c:simultaneous_coupling_2}$, interval $I \subset \ZZ$, densities $\rho \leq \rho' \in [\rho_{-},\rho_{+}]$ and $ \epsilon \in (0,1]$ such that $\rho-\epsilon \geq \rho_{-}$, there exists a coupling between two pairs of zero range processes $(\eta_{s}, \eta'_{s})_{s \geq 0}$ and $(\bar{\eta}_{s}, \bar{\eta}'_{s})_{s \geq 0}$ such that
\begin{enumerate}
\item[1.] $(\eta_{s}, \eta'_{s})_{s \geq 0}$ is distributed as $\PP_{\rho, \rho'}$ and $(\bar{\eta}_{s}, \bar{\eta}'_{s})_{s \geq 0}$ is distributed as $\PP_{\rho-\epsilon, \rho+\epsilon}$;
\item[2.] $(\bar{\eta}_{s}, \bar{\eta}'_{s})_{s \geq 0}$ is independent from $(\eta_{0}, \eta_{0}')$;
\item[3.]
\begin{equation}\label{eq:bad_simultanous_counpling}
\PP\left[
\begin{array}{cl}
\text{there exists } x \in I \text{ such that} \\ \eta_{t}(x) < \bar{\eta}_{t}(x) \text{ or } \eta_{t}'(x) > \bar{\eta}'_{t}(x)
\end{array}
\right] \leq \useconstant{c:simultaneous_coupling} t (|I|+t)e^{-\useconstant{c:simultaneous_coupling}^{-1}\epsilon^{2}t^{\sfrac{1}{4}}}.
\end{equation}
\end{enumerate}
\end{prop}

\par The construction of the coupling stated in the proposition above is similar to the one in Proposition~\ref{prop:coupling}. Hence, in the proof presented here we only point out the main differences between the constructions.

\begin{proof}
We want to couple two pairs $(\eta_{s}, \eta'_{s})_{s \geq 0}$ and $(\bar{\eta}_{s}, \bar{\eta}'_{s})_{s \geq 0}$ with respective densities $(\rho, \rho')$ and $(\rho-\epsilon, \rho'+\epsilon)$. We also start with two independent pairs of configurations and two copies of the graphical construction of Subsection~\ref{subsec:graphical}, $\mathscr{P}_{1}=(\mathcal{P}_{1}(x))_{x \in \ZZ}$ and $\mathscr{P}_{2}=(\mathcal{P}_{2}(x))_{x \in \ZZ}$.

The pair $(\bar{\eta}_{s}, \bar{\eta}'_{s})_{s \geq 0}$ will evolve with the second copy of the graphical construction $\mathscr{P}_{2}$, up to change of heights in the piles. We then need to set the evolution of $(\eta_{s}, \eta'_{s})_{s \geq 0}$ so that $\bar{\eta}_{t} \leq \eta_{t}$ and $\eta'_{t} \leq \bar{\eta}'_{t}$ inside $I$ with high probability. This will also use the pairing between the configuration and the matching times. We will split the evolution in two parts. First, we obtain $\bar{\eta}_{t} \leq \eta_{t}$ inside $I$ with high probability. When this is done, we continue the construction to ensure the other domination. For the first half of the matching times, we only pair $\eta$ to $\bar{\eta}$ and use the evolution of the coupling from Proposition~\ref{prop:coupling}. This gives $\bar{\eta}_{t} \leq \eta_{t}$ inside $I$ with high probability.

Once this is complete, we try to get $\eta'_{t} \leq \bar{\eta}'_{t}$. This is done using the second half of the matching times. In this case, the matching also includes the particles from the processes $\eta'$ and $\bar{\eta}'$. The coupling is still the same one from Proposition~\ref{prop:coupling}, but we need to be more careful, due to the existence of the particles from $\eta$ and $\bar{\eta}$. Whenever a particle jumps to a new site, we may need to perform a change of the matching. We update the pairing to obey the rules $\eta \leq \eta'$ and $\bar{\eta} \leq \bar{\eta}'$. When a particle jumps, it goes to its correct place in the new pile. If it meets its pair or it is a particle from $\eta$ or $\bar{\eta}$, we update the matching just by changing the heights of the matched particles. This will also grantee that particles that already meet stay together. Figure~\ref{fig:changin_heights} gives an example where an update is necessary.

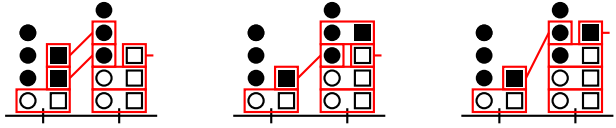
\begin{figure}\label{fig:changin_heights}
\begin{center}
\begin{tikzpicture}

\draw[thick] (-4,0)--(-2,0);
\draw[thick] (-1,0)--(1,0);
\draw[thick] (2,0)--(4,0);
\foreach \x in {-3.5,-2.5,-0.5, 0.5, 2.5, 3.5}
		\draw[thick](\x,-0.1)--(\x,0.1);


\foreach \x in {-3.5,-0.5, 2.5}
	{
		\draw[thick] (\x-0.2, 0.2) circle (0.1);
		\draw[thick] (\x+0.8, 0.2) circle (0.1);
		\draw[thick] (\x+0.8, 0.5) circle (0.1);
		\fill[thick]  (\x-0.2, 0.5) circle (0.11);
		\fill[thick]  (\x-0.2, 0.8) circle (0.11);
		\fill[thick]  (\x-0.2, 1.1) circle (0.11);
		\fill[thick]  (\x-0.2, 1.1) circle (0.11);
		\fill[thick] (\x+0.8, 0.8) circle (0.11);
		\fill[thick] (\x+0.8, 1.1) circle (0.11);
		\fill[thick] (\x+0.8, 1.4) circle (0.11);
		\draw[thick] (\x+1.1, 0.1) rectangle (\x+1.3, 0.3);
		\draw[thick] (\x+1.1, 0.4) rectangle (\x+1.3, 0.6);
		\draw[thick] (\x+1.1, 0.7) rectangle (\x+1.3, 0.9);
		\draw[thick] (\x+0.1, 0.1) rectangle (\x+0.3, 0.3);
		\fill (\x+0.09, 0.39) rectangle (\x+0.31, 0.61);
	}
	
\fill (-3.5+0.09, 0.69) rectangle (-3.5+0.31, 0.91);
\fill (0.59, 0.99) rectangle (0.81, 1.21);
\fill (3.59, 0.99) rectangle (3.81, 1.21);

\foreach \x in {-3.5,-0.5,2.5}
	{
		\draw[red, thick] (\x-0.35,0.35) rectangle (\x+0.35, 0.05);
		\draw[red, thick] (\x+0.65,0.35) rectangle (\x+1.35, 0.05);
		\draw[red, thick] (\x+0.65,0.35) rectangle (\x+1.35, 0.65);
	}

\draw[red, thick] (-3.5+0.05,0.35) rectangle (-3.5+0.35, 0.65);
\draw[red, thick] (-3.5+0.65,0.65) rectangle (-3.5+0.95, 0.95);
\draw[red, thick] (-3.5+0.35,0.5) -- (-3.5+0.65, 0.8);
\draw[red, thick] (-0.5+0.05,0.35) rectangle (-0.5+0.35, 0.65);
\draw[red, thick] (-0.5+0.65,0.65) rectangle (-0.5+0.95, 0.95);
\draw[red, thick] (-0.5+0.35,0.5) -- (-0.5+0.65, 0.8);
\draw[red, thick] (-3.5+0.05,0.65) rectangle (-3.5+0.35, 0.95);
\draw[red, thick] (-3.5+0.65,0.95) rectangle (-3.5+0.95, 1.25);
\draw[red, thick] (-3.5+0.35,0.8) -- (-3.5+0.65, 1.1);
\draw[red, thick] (0.15,0.95) rectangle (0.85, 1.25);
\draw[red, thick] (3.15,0.65) rectangle (3.85, 0.95);
\draw[red, thick] (2.5+0.05,0.35) rectangle (2.5+0.35, 0.65);
\draw[red, thick] (2.5+0.65,0.95) rectangle (2.5+0.95, 1.25);
\draw[red, thick] (2.5+0.35,0.5) -- (2.5+0.65, 1.1);

\draw[red, thick] (-3.5+1.05,0.65) rectangle (-3.5+1.35, 0.95);
\draw[red, thick] (-0.5+1.05,0.65) rectangle (-0.5+1.35, 0.95);
\draw[red, thick] (2.5+1.05,0.95) rectangle (2.5+1.35, 1.25);
\draw[red, thick] (-3.5+1.35, 0.8) -- (-3.5+1.45, 0.8);
\draw[red, thick] (-0.5+1.35, 0.8) -- (-0.5+1.45, 0.8);
\draw[red, thick] (2.5+1.35, 1.1) -- (2.5+1.45, 1.1);

\end{tikzpicture}
\caption{A pairing where an update is necessary. Notice that, after the jump, in order to obey that particles from density $\rho$ stay always below particles from the configuration with density $\rho'$, we change the pairing in the pile.}
\end{center}
\end{figure}

It is easy to verify that all the estimates in the proof of Proposition~\ref{prop:coupling} remain valid in this case, up to a change of constants.
\end{proof}

\subsection{The box notation}\label{subsec:box_notation_2}
~
\par We now begin to introduce the notation for the proof of Theorem~\ref{teo:positive_velocity}. Some notation was already introduced in Subsection~\ref{subsec:box_notation} and we recall it here too.

\par In this subsection we write $I_{k}=\left[-\frac{\ell_{k}}{4}L_{k}^{2}, \frac{\ell_{k}}{4}L_{k}^{2}\right]$ and, for $m =(x,sL_{k}) \in \ZZ \times L_{k}\NN_{0}$, let $I_{k}(m)=x+I_{k}$.

\par We say that a path $\gamma:[0,L_{k}] \to \ZZ$ is $\eta$-allowed (for the scale $k$) if
\begin{enumerate}
\item $\gamma(0)=0$;
\item $\gamma(t) \in I_{k}$, for all $t \in [0, L_{k}]$;
\item $\gamma$ is a nearest-neighbor path;
\item $\gamma$ only moves when a particle of $\eta$ jumps from that site.
\end{enumerate}
Being $\eta$-allowed is a non-decreasing property. This means that if $\gamma$ is $\eta$-allowed and $\eta \preccurlyeq \tilde{\eta}$, then $\gamma$ is also $\tilde{\eta}$-allowed.

\par With high probability, the front of the infection, $r_{t}$, is an $\eta$-allowed path. In order to prove that it moves to the right with positive speed, we will verify that it shares a site with two or more particles a positive proportion of time. In these times, $r_{t}$ has a drift to the right. However, instead of investigating directly these times, we introduce a quantity that measures the amount of time a path is within distance $R$ from at least two particles. For $R>0$, $t > 0$ and a c\`adl\`ag path $\gamma:[0,t] \to \ZZ$, let
\begin{equation}
V^{R,t}_{\eta}(\gamma)=\frac{1}{t}\left| \left\{
s \in [0, t] : \sum_{x = -\gamma(s)-R}^{\gamma(s)+R} \eta_{s}(x) \geq 2 \right\}\right|,
\end{equation}
where $|A|$ denotes the Lebesgue measure of the set $A$.

\par The bad event we are interested in here deals with the existence of an $\eta$-allowed path $\gamma$ with $V_{\eta'}^{R, L_{k}}(\gamma)$ small.

\par Observe that, if $\eta \preccurlyeq \tilde{\eta}$, then $V_{\eta}^{R,k}(\gamma) \leq V_{\tilde{\eta}}^{R,k}(\gamma)$. This will allow us to use the stronger version of the zero range process decoupling, Proposition~\ref{prop:simultaneous_decoupling}.

\par In a similar flavor of~\eqref{eq:velocities}, we introduce the sequence
\begin{equation}\label{eq:density_occupied_sites}
\epsilon_{0}=\epsilon>0 \qquad \text{and} \qquad \epsilon_{k+1}=\epsilon_{k}\left(1-\frac{1}{(k+1)^{2}}\right).
\end{equation}
Observe that the sequence above is non-increasing and $\epsilon_{\infty}= \lim \epsilon_{k}$ is positive. Consider the sequence of events
\begin{equation}\label{eq:renorm_2_event}
F_{k}^{R}=\left\{
(\eta, \eta'): \begin{array}{cl} \text{there exists a path $\gamma$ that is} \\ \text{$\eta'$-allowed for scale $k$ and } V_{\eta}^{R,L_{k}}(\gamma) \leq \epsilon_{k}
\end{array}
\right\}
\end{equation}
The events $F_{k}^{R}$ are non-increasing in $\eta$ and non-decreasing in $\eta'$. Besides, when $R \leq \frac{3\ell_{k}}{4}L_{k}$ the event $F_{k}^{R}$ has support in $B_{k}$, defined in~\eqref{eq:renormalisation_box}. For $m =(x, sL_{k}) \in \ZZ \times L_{k} \NN_{0}$, denote by $F_{k}^{R}(m)$ the translation of the event $F_{k}^{R}$ by the vector $m$.

\par For some fixed $\rho_{0}>0$, recall we defined the sequence $(\rho_{k})_{k \in \NN_{0}}$ in~\eqref{eq:density_scales} by setting $\rho_{k}=\rho_{k+1}(1+L_{k}^{-\sfrac{1}{16}})$. We set $\rho_{0}'=\rho_{0}$ and define
\begin{equation}\label{eq:density_scales_2}
\rho_{k+1}'=\rho_{k}'(1+L_{k}^{-\sfrac{1}{16}}).
\end{equation}
In this case, $(\rho'_{k})_{k \in \NN_{0}}$ is increasing and $\rho'_{\infty}=\lim \rho'_{k}$ exists and is finite.

\par Finally, define the probabilities
\begin{equation}\label{eq:q_k}
q_{k}=\PP_{\rho_{k}, \rho_{k}'}[F_{k}^{R}].
\end{equation}

\subsection{Estimates on $q_{k}$}\label{subsec:renorm_2}

\par We now focus on the bounds of $q_{k}$. This will be done in a similar way as in Subsection~\ref{subsec:renorm}, and hence some proofs are omitted.

\par The first thing we need to do is to relate the properties of being $\eta'$-allowed for different scales. We prove a lemma that bounds the probability of the following event
\begin{equation}
G_{k}=\left\{\begin{array}{cl}\text{all paths } \gamma \text{ that are } \eta \text{-allowed for the scale} \\ k+1 \text{ do not leave } I_{k} \text{ before time } L_{k}
\end{array}\right\}.
\end{equation}

\newconstant{c:bound_G}

\begin{lemma}\label{lemma:renorm_2.1}
There exists a positive constant $\useconstant{c:bound_G}=\useconstant{c:bound_G}(\rho_{-}, \rho_{+})$ such that, for all $\rho \in [\rho_{\infty}, \rho'_{\infty}]$ and $k \geq 0$, we have
\begin{equation}
\PP_{\rho}[G_{k}^{c}] \leq \useconstant{c:bound_G}e^{-\useconstant{c:bound_G}^{-1}L_{k}}.
\end{equation}
\end{lemma}

\begin{proof}
We consider two paths that are $\eta$-allowed in the scale $k+1$: $\gamma_{+}$, that always jumps to the right, and $\gamma_{-}$, that always jumps to the left. Observe that, if $X \sim \poisson(\Gamma_{+} L_{k} A(L_{k},L_{k}))$, with $A(\cdot, \cdot)$ as in~\eqref{lemma:many_particles}, then, for all $k$ large,
\begin{equation}
\begin{split}
\PP_{\rho}[G_{k}^{c}]& \leq \PP_{\rho}\left[\begin{array}{cl} \gamma_{+} \text{ or } \gamma_{-} \text{ leaves } I_{k} \\
\text{before time } L_{k}
\end{array}\right] \\
& \leq L_{k}^{3}\PP_{\rho}\left[\begin{array}{cl}
\eta_{s}(0) \geq A(L_{k}, L_{k}), \\ \text{for some } s \in [0,L_{k}]
\end{array}\right]\\
& \qquad +2\PP\left[X \geq \frac{\ell_{k}}{4} L_{k}^{2}\right] \\
& \leq \useconstant{c:quadratic_displacement}(L_{k}^{3}+1)e^{-\useconstant{c:many_particles}^{-1}L_{k}}+e^{-cL_{k}} \leq \useconstant{c:bound_G}e^{-\useconstant{c:bound_G}^{-1}L_{k}}.
\end{split}
\end{equation}
By possibly increasing the value of $\useconstant{c:bound_G}$, we obtain that the estimate above is true for all $k \geq 0$ and conclude the proof.
\end{proof}

\par For $m =(x, sL_{k}) \in \ZZ \times L_{k}\NN_{0}$, if we define the translation
\begin{equation}
G_{k}(m)=\left\{\begin{array}{cl}\text{all paths } \gamma \text{ that are } \eta \text{-allowed for scale } k+1 \text{ and touch } m \\ \text{satisfy that } \gamma|_{[sL_{k}, (s+1)L_{k}]} \text{ does not leave } I_{k}(m)
\end{array}\right\},
\end{equation}
we easily obtain the bound $\PP_{\rho}[G_{k}(m)] \leq \PP_{\rho}[G_{k}]$.

\newconstant{c:renorm_2}

\par We focus now on the probabilities $q_{k}$. As before, the first step is to obtain a recursive inequality that relates $q_{k}$ and $q_{k+1}$.
\begin{lemma}\label{lemma:renorm_2.2}
There exists $k_{0}$ such that, for all $k \geq k_{0}$ and  $ 1 \leq R \leq \frac{3\ell_{k}}{4}L_{k}$,
\begin{equation}
q_{k+1}\leq \useconstant{c:renorm_2}L_{k+1}^{28}[q_{k}^{4}+e^{-\useconstant{c:renorm_2}^{-1}\rho_{\infty}^{2}L_{k}^{\sfrac{1}{8}}}].
\end{equation}
\end{lemma}

\begin{proof}
The proof is very similar as the one of Lemma~\ref{lemma:renorm_1}, but we use the stronger version of the decoupling in this case. Here, we only prove that the events $F_{k}^{R}$ are cascading.

Fix $k_{0} \in \NN_{0}$ such that, for all $k \geq k_{0}$
\begin{equation}
\frac{1}{6(k+1)^{2}} \geq \frac{1}{L_{k}^{2}}.
\end{equation}

Fix $k \geq k_{0}$, a value $ 1 \leq R \leq \frac{\ell_{k}}{2}L_{k}$ and assume we are in $F_{k+1}^{R}$. We claim that
\begin{equation*}
\begin{array}{cl}
\text{either } G_{k}(m)^{c} \text{ holds for some } m \in M_{k} \text{ or there are} \\
\text{seven elements } m_{i}=(x_{i},s_{i})\in M_{k}, 1 \leq i \leq 7, \text{ with}\\
s_{i} \neq s_{j}, \text{ if } i \neq j, \text{ such that } F_{k}^{R}(m) \text{ holds}.
\end{array}
\end{equation*}

Once again, the proof follows by contradiction. Assume we are in the event $F_{k+1}^{R}$, that $G_{k}(m)$ holds for all $m \in M_{k}$, and that $F_{k}^{R}(m)$ holds for at most six values of $m \in M_{k}$ with different time coordinates.

Observe that, if $F_{k+1}^{R}$ holds, there exists an $\eta'_{k+1}$-allowed path $\gamma$ with $V^{R,L_{k+1}}_{\eta_{k+1}}(\gamma) \leq \epsilon_{k+1}$. Besides, for all but at most six values of $ 0 \leq s \leq L_{k}^{2}$, the path $\gamma_{s}=\gamma|_{[sL_{k}, (s+1)L_{k}]}$ is $\eta'_{k}$-allowed and $V^{R,L_{k}}_{\eta_{k}}(\gamma_{s}) > \epsilon_{k}$. Observe now that
\begin{equation}
\begin{split}
V^{R,L_{k+1}}_{\eta_{k+1}}(\gamma) & =\frac{L_{k}}{L_{k+1}}\sum_{s=0}^{L_{k}^{2}-1}V^{R,L_{k}}_{\eta_{k+1}}(\gamma_{s}) \\
& \geq \frac{L_{k}}{L_{k+1}}\sum_{s=0}^{L_{k}^{2}-1}V^{R,L_{k}}_{\eta_{k}}(\gamma_{s}) \\
& > \frac{L_{k}}{L_{k+1}}\epsilon_{k}(L_{k}^{2}-6) \\
& \geq \epsilon_{k}\left(1-\frac{6}{L_{k}^{2}}\right) \geq \epsilon_{k+1},
\end{split}
\end{equation}
a contradiction.
\end{proof}

\par Observe that Lemma~\ref{lemma:renorm_2} is also valid for the quantities $q_{k}$. The proof remains the same and we omit it here.
\begin{lemma}\label{lemma:renorm_2.3}
There exists $k_{1} \geq k_{0}$ such that, for $k \geq k_{1}$ and any choice of $v_{0}$, if
\begin{equation}\label{eq:renorm_trigger}
q_{k} \leq e^{-\log^{\sfrac{5}{4}}L_{k}},
\end{equation}
then
\begin{equation}
q_{k+1} \leq e^{-\log^{\sfrac{5}{4}}L_{k+1}}.
\end{equation}
\end{lemma}

\par We now prove an analogous of Lemma~\ref{lemma:renorm_3}: we will verify that, if $\epsilon_{0}$ is small enough and $R$ and $k$ are large enough, then~\eqref{eq:renorm_trigger} holds.

\begin{lemma}\label{lemma:renorm_2.4}
There exists $k_{2} \geq k_{1}$, $R \leq \frac{\ell_{k_{2}}}{2} L_{k_{2}}$ and $\epsilon_{0}>0$ such that $q_{k_{2}} \leq e^{-\log^{\sfrac{5}{4}}L_{k_{2}}}$.
\end{lemma}

\begin{proof}
First we compute
\begin{equation}
\begin{split}
\PP_{\rho}\left[\sum_{-\frac{\ell_{k}}{2} L_{k} \leq x \leq \frac{\ell_{k}}{2} L_{k}}\eta_{0}(x) \leq 1\right] & \leq \PP_{\rho}\left[\begin{array}{cl}
\text{there exists } x \in I_{k} \text{ such that } \eta_{0}(y)=0, \\ \text{for all } y \in [-\frac{\ell_{k}}{2} L_{k}, \frac{\ell_{k}}{2} L_{k}] \setminus \{x\}
\end{array}\right]\\
& \leq L_{k}^{\sfrac{3}{2}}\PP_{\rho}[\eta_{0}(0)=0]^{\ell_{k} L_{k}} \leq e^{-cL_{k}}.
\end{split}
\end{equation}

Now define, for $R_{k}=\frac{\ell_{k}}{2} L_{k}$,
\begin{equation}
\tilde{F}_{k}=\left\{
(\eta, \eta'): \begin{array}{cl} \text{there exists a path $\gamma$ that is} \\ \text{$\eta'$-allowed for scale $k$ and } V^{R_{k},L_{k}}_{\eta}(\gamma) =0
\end{array}
\right\},
\end{equation}
and observe that
\begin{equation}
\PP_{\rho_{k}, \rho_{k}'}[\tilde{F}_{k}] \leq \PP_{\rho_{\infty}, \rho_{\infty}'}[\tilde{F}_{k}] \leq \PP_{\rho_{\infty}}\left[\sum_{-\frac{\ell_{k}}{2} L_{k} \leq x \leq \frac{\ell_{k}}{2} L_{k}}\eta_{0}(x) \leq 1\right] \leq e^{-cL_{k}}.
\end{equation}

Fix $k_{2} \geq k_{1}$ such that $2e^{-cL_{k_{2}}} \leq e^{-\log^{\sfrac{5}{4}}L_{k_{2}}}$. Since $\lim_{\epsilon_{k_{2}} \to 0}
\PP_{\rho_{k_{2}}, \rho_{k_{2}}'}[F_{k_{2}}^{R_{k_{2}}}]= \PP_{\rho_{k_{2}}, \rho_{k_{2}}'}[\tilde{F}_{k_{2}}]$, we can choose $\epsilon_{k_{2}}$ such that $\PP_{\rho_{k_{2}}, \rho_{k_{2}}'}[F_{k_{2}}^{R_{k_{2}}}] \leq \PP_{\rho_{k_{2}}, \rho_{k_{2}}'}[\tilde{F}_{k_{2}}]+e^{-cL_{k_{2}}}$ and conclude that
\begin{equation}
\PP_{\rho_{k_{2}}, \rho_{k_{2}}'}[F_{k_{2}}^{R_{k_{2}}}] \leq \PP_{\rho_{k_{2}}, \rho_{k_{2}}'}[\tilde{F}_{k_{2}}]+e^{-cL_{k_{2}}} \leq 2e^{-cL_{k_{2}}}.
\end{equation}
This concludes the proof with $R=R_{k_{2}}=\frac{\ell_{k_{2}}}{2} L_{k_{2}}$ and the suitable choice of~$\epsilon_{0}$.
\end{proof}

\subsection{Proof of Theorem~\ref{teo:positive_velocity}}
~
\par We now turn to the proof that $r_{t}$ travels to the right with positive velocity. Our first goal is to obtain bounds for the sequence of times $L_{k}$. The renormalisation developed in Subsection~\ref{subsec:renorm} will be used in this step, since it says that, considering the process stopped at any of these times, the wave front contains at least 2 particles during a positive proportion of time. Once this is done, we use a concatenation argument similar to the one used in the proof of Theorem~\ref{teo:finite_velocity} to conclude.

\newconstant{c:concentration_2}

\par We begin by introducing the zero-mean martingale
\begin{equation}\label{eq:martingale}
M_{t}=r_{t}-r_{0}-\int_{0}^{t}\frac{1}{2}g(\eta_{s}(r_{s})) \charf{\{\eta_{s}(r_{s}) \geq 2\}} \,\mathrm{d}s,
\end{equation}
and stating a concentration estimate for it.
\begin{prop}\label{prop:martingale_large_deviation}
For every $\delta>0$, there exists a positive constant $\useconstant{c:concentration_2}$ that depends also on $\rho>0$ such that, for all $k$,
\begin{equation}
\PP_{\rho}[|M_{L_{k}}| \geq \delta L_{k}] \leq \useconstant{c:concentration_2}e^{-\useconstant{c:concentration_2}^{-1}L_{k}^{\sfrac{1}{8}}}.
\end{equation}
\end{prop}

\par We postpone the proof of this proposition to the Appendix. With it, we can study the behavior of $r_{t}$ at the times $L_{k}$. Since we know that $M_{L_{k}}$ is concentrated around its mean, in order to verify that $r_{L_{k}}$ drifts to the right it suffices to study the integral term in~\eqref{eq:martingale}.

\begin{prop}\label{prop:drift}
There exists $k_{3} \geq k_{2}$ and $\delta>0$ such that, for all $k \geq k_{3}$,
\begin{equation}
\PP_{\rho_{\infty}}[r_{L_{k}} \leq \delta L_{k}, \text{ and } r_{0}=0] \leq 4e^{-\log^{\sfrac{5}{4}}L_{k}}. 
\end{equation}
\end{prop}

\par The idea of the proof is to use that, with high probability, the path $r_{t}$ is $\eta$-allowed. Therefore, for a positive fraction of times, there are more than two particles close to it. Using this fact, we will prove that there is a positive fraction of times for which two particles are on top of the front, producing a drift to the right.

\begin{proof}
Begin by introducing the event
\begin{equation}
\bar{G}_{k}=\left\{ \sup _{0 \leq t \leq L_{k}}|r_{t}| \geq \frac{\ell_{k}}{4}L_{k}^{2}\right\},
\end{equation}
and notice that, by Lemmas~\ref{lemma:infection_1} and~\ref{lemma:infection_3},
\begin{equation}
\PP_{\rho_{\infty}}\left[\begin{array}{cl}
(r_{t})_{0 \leq t \leq L_{k}} \text{ is not } \eta \text{-allowed for} \\ \text{ the scale } k \text{ and } r_{0}=0
\end{array}\right] \leq \PP_{\rho_{\infty}}[\bar{G}_{k}, r_{0}=0] \leq ce^{-c^{-1}L_{k}},
\end{equation}
for some positive constant $c$.

By possibly increasing the value of $k_{3}$, we obtain, for $k \geq k_{3}$,
\begin{equation}\label{eq:part_close_to_r}
\PP_{\rho_{\infty}}[V^{R,L_{k}}_{\eta}(r_{t}) \leq \epsilon_{\infty}, r_{0}=0] \leq q_{k}+\PP_{\rho_{\infty}}[\bar{G}_{k}, r_{0}=0] \leq 2e^{-\log^{\sfrac{5}{4}}L_{k}}.
\end{equation}
In the equation above and in what follows, we use a slight abuse of notation by denoting $r_{t}$ as the whole path $(r_{t})_{t \geq 0}$ or its restriction when convenient.

We now claim that, if, for some time $t \in [0, L_{k}]$, we have $\sum_{x = r_{t}-R}^{r_{t}+R}\eta_{t}(x) \geq 2$, then there exists a positive probability that
\begin{equation}\label{eq:density_event}
\left| \left\{s \in [t, t+1]: \eta_{s}(r_{s}) \geq 2\right\}\right| \geq \delta',
\end{equation}
for some $\delta'>0$, where $|A|$ denotes the Lebesgue measure of $A$.

One way to verify this claim is by contradiction. If the probability of the event above is zero for every $\delta'>0$, then, by taking the limit as $\delta' \to 0$, the probability that the front of the infection has at least two particles between times $t$ and $t+1$ is zero. This, however, contradicts the fact that, at time $t$, there is at least one particle at distance at most $R$ from the front and that this particle has a positive chance of reaching the front before time $t+1$.

This implies that, conditioned on $(\eta_{t}, r_{t})$, the indicator function of the event in~\eqref{eq:density_event} stochastically dominates a random variable $X$ with positive expectation and that assumes only the values zero and one. Define $\delta=\frac{\epsilon_{\infty}}{4}\EE_{\rho_{\infty}}[X]$.

We now investigate the event $\{V^{R, L_{k}}_{\eta}(r_{t}) \geq \epsilon_{\infty}\}$. In it, there exists a sequence of times $(t_{i})_{i \in [N]}$, $N =\lfloor \frac{\epsilon_{\infty}}{2}L_{k} \rfloor$, such that $|t_{i}-t_{j}| \geq 2$, for $i \neq j$, and $\sum_{x = r_{t_{i}}-R}^{r_{t_{i}}+R}\eta_{t_{i}}(x) \geq 2$, for all $i \in [N]$. These times allow us to estimate
\begin{equation}
\begin{split}
\PP_{\rho_{\infty}}[V^{0, L_{k}}_{\eta}(r_{t}) \leq \delta\delta', r_{0}=0] & \leq \PP_{\rho_{\infty}}\left[X_{1}+X_{2} + \cdots + X_{N} \leq \delta L_{k}\right] \\
& \qquad + \PP_{\rho_{\infty}}[V^{R, L_{k}}_{\eta}(r_{t}) \leq \epsilon_{\infty}, r_{0}=0],
\end{split}
\end{equation}
where $(X_{i})_{i \in [N]}$ are iid copies of $X$.

Recall that $\delta=\frac{\epsilon_{\infty}}{4}\EE_{\rho_{\infty}}[X]$ and $N =\lfloor \frac{\epsilon_{\infty}}{2}L_{k} \rfloor$ and write $\bar{\delta}=\delta\delta'$. Standard concentration bounds for $(X_{i})_{i \in [N]}$ and Equation~\eqref{eq:part_close_to_r} imply
\begin{equation}
\PP_{\rho_{\infty}}[V^{0, L_{k}}_{\eta}(r_{t}) \leq \bar{\delta}, r_{0}=0] \leq 3e^{-\log^{\sfrac{5}{4}}L_{k}}.
\end{equation}

Notice that, if $V^{0, L_{k}}_{\eta}(r_{t}) \geq \bar{\delta}$, then
\begin{equation}
\int_{0}^{L_{k}}\frac{1}{2}g(\eta_{s}(r_{s})) \charf{\{\eta_{s}(r_{s}) \geq 2\}} \,\mathrm{d}s \geq \frac{\bar{\delta}}{2}g(2)L_{k}.
\end{equation}
Recall\eqref{eq:martingale}, set $\bar{\delta}'=\frac{\bar{\delta}}{2}g(2)$ and use Proposition~\ref{prop:martingale_large_deviation} to conclude that
\begin{equation}
\begin{split}
\PP_{\rho_{\infty}}[r_{L_{k}} \leq \bar{\delta}'L_{k}, r_{0}=0] & \leq \PP_{\rho_{\infty}}[V^{0, L_{k}}_{\eta}(r_{t}) \leq \bar{\delta}, r_{0}=0] + \PP_{\rho_{\infty}}[|M_{L_{k}}| \geq \bar{\delta}'L_{k}] \\ 
& \leq 4e^{-\log^{\sfrac{5}{4}}L_{k}}.
\end{split}
\end{equation}
\end{proof}

We are now ready to conclude the proof of Theorem~\ref{teo:positive_velocity}. The last step is a concatenation argument similar to the one used in the last section to conclude Theorem~\ref{teo:finite_velocity}. For this reason, we provide just a sketch of the proof.

\begin{proof}[Proof of Theorem~\ref{teo:positive_velocity}]
Let $k_{3}$ be as in Proposition~\ref{prop:drift} and notice that we may assume $L \geq 2L_{k_{3}+2}$. Choose $\bar{k} \geq k_{3}$ such that
\begin{equation}
2L_{\bar{k}+2} \leq L < 2L_{\bar{k}+3}.
\end{equation}

For $m =(x,s) \in \ZZ \times L_{k}\NN_{0}$, define the events
\begin{equation}
\bar{E}_{k}(m)=\{r_{L_{k}}(m)-x \leq \delta L_{k} \text{ and } r_{0}(m)=x\},
\end{equation}
where $\delta$ is given by Proposition~\ref{prop:drift}. Consider also
\begin{equation}
\bar{H}_{k}(m)=\left\{x-\inf_{0 \leq t \leq L_{k}}r_{t}(m) \leq 2(\Gamma_{+}+1)L_{k}\right\}.
\end{equation}
Finally, define
\begin{equation}
A = \left\{\begin{array}{cl}
r_{t} \geq v_{+}t+L, \\
\text{for some } t \geq 0
\end{array}\right\},
\end{equation}
where $v_{+}$ is given by Theorem~\ref{teo:finite_velocity} and is such that~\eqref{eq:finite_speed} holds.

\newconstant{c:bad_event_bound_2}

Define the set of indices
\begin{equation}
\bar{M}_{k}=\{m \in \ZZ \times L_{k} \NN_{0}: B_{k}(m) \cap B_{k+2} \neq \emptyset\},
\end{equation}
and consider the event
\begin{equation}
\bar{B}_{\bar{k}}=A^{c} \cap \bigcap_{k \geq \bar{k}}\bigcap_{m \in \bar{M}_{k}} \bar{E}_{k}(m)^{c} \cap \bar{H}_{k}(m).
\end{equation}
Proposition~\ref{prop:drift}, Lemma~\ref{lemma:infection_3} and Theorem~\ref{teo:finite_velocity} imply, by possibly changing constants, that
\begin{equation}
\PP_{\rho_{\infty}}[\bar{B}_{\bar{k}}^{c}] \leq \useconstant{c:bad_event_bound_2}e^{-\useconstant{c:bad_event_bound_2}^{-1}\log^{\sfrac{5}{4}}L}.
\end{equation}

Similarly to the proof of Theorem~\ref{teo:finite_velocity}, define
\begin{equation}
\bar{J}_{\bar{k}}=\bigcup_{k \geq \bar{k}}\bigcup_{\ell=0}^{\sfrac{L_{k+2}}{L_{k}}}\{\ell L_{k}\}.
\end{equation}
On the event $\bar{B}_{\bar{k}} \cap \{r_{0}=0\}$, induction implies that
\begin{equation}
r_{t} \geq \delta t, \qquad \text{for all } t \in J_{\bar{k}}.
\end{equation}

We now interpolate for the remaining values of $t$. Consider initially $t \geq L$. Let $\kappa$ be the smallest $k \geq \bar{k}$ such that
\begin{equation}
\ell L_{\kappa} \leq t < (\ell+1)L_{\kappa}, \qquad \text{for some } \ell < \frac{L_{\kappa+2}}{L_{\kappa}}.
\end{equation}
Let $\bar{\ell}$ denote the unique value of $\ell$ and observe that $\ell \geq L_{\kappa}/L_{\kappa-1}$. This easily implies, by increasing the value of $L$ if necessary,
\begin{equation}
r_{t} \geq \delta \ell L_{\kappa}-(2\Gamma_{+}+1)L_{\kappa} \geq \frac{\delta}{2}t. 
\end{equation}

The interpolation for the values $t \leq L$ is done in the same way as in the proof of Theorem~\ref{teo:finite_velocity} and we omit it here.
\end{proof}

\par To conclude, we present the proof of Proposition~\ref{prop:horizontal_decoupling}.

\begin{proof}[Proof of Proposition~\ref{prop:horizontal_decoupling}.]
Observe first that we may increase the side-length of both boxes by at most $\dist_{V}+s$ and assume the boxes have the form
\begin{align*}
B_{1}=[-s,0]\times[0, s], \\
B_{2}=[\dist_{H},\dist_{H}+s]\times[0,s].
\end{align*}

Figure~\ref{fig:boxes_horizontal} can be used as a reference.

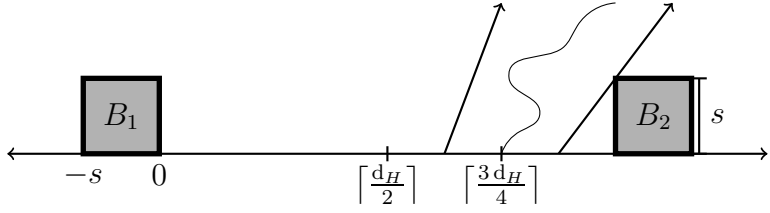
\begin{figure}[h]\label{fig:boxes_horizontal}
\begin{center}
\begin{tikzpicture}

\draw[<->, thick]  (-3,0) -- (7,0);

\fill [black!30!](-2,1) rectangle (-1,0);
\draw [line width=2pt] (-2,1) rectangle (-1,0);
\node at (-1.5,0.5) {$B_{1}$};

\fill [black!30!](5,0) rectangle (6,1);
\draw [line width=2pt] (5,0) rectangle (6,1);
\node at (5.5,0.5) {$B_{2}$};

\draw[thick] (6.1,0)--(6.1,1);
\draw[thick] (6,1)--(6.2,1);
\draw[thick] (6,0)--(6.2,0);
\node[right] at (6.1,0.5) {$s$};

\node[below] at (-2,0) {$-s$};
\node[below] at (-1,0) {$0$};
\draw[thick] (2,-0.1)--(2,0.1);
\node[below] at (2,0){$ \left\lceil \frac{\dist_{H}}{2} \right\rceil$};

\draw[thick] (3.5,-0.1)--(3.5,0.1);
\node[below] at (3.5,0){$ \left\lceil \frac{3\dist_{H}}{4} \right\rceil$};
\draw[->,thick] (4.25,0)--(5.75,2);
\draw[->, thick](2.75,0)--(3.5,2);
\draw (3.5,0) to [curve through={(3.75,0.3) .. (4, 0.5) .. (3.6, 1) .. (4.3, 1.4) .. (4.5, 1.7)}] (5,2);

\end{tikzpicture}
\caption{The boxes $B_{1}$ and $B_{2}$. Notice also the infection process $r_{t}(\lceil\sfrac{3d_{H}}{4}\rceil,0)$ and the lines that bound the evolution of the front.}
\end{center}
\end{figure}

We now verify that, with high probability, the outcomes of $f_{1}$ and $f_{2}$ are determined by disjoint parts of the graphical construction in the space-time.

Consider initially $r_{t}(\lceil\sfrac{3d_{H}}{4}\rceil,0)$. Observe that, if $f_{2}$ is not determined by the graphical construction restricted to $(\left\lceil\sfrac{\dist_{H}}{2} \right\rceil, \infty) \times [0,s]$ (and the initial configuration restricted to $(\left\lceil\sfrac{\dist_{H}}{2} \right\rceil, \infty)$), then the infection $r_{t}(\lceil\sfrac{3d_{H}}{4}\rceil,0)$ either touches $B_{2}$ or the line $y=\left\lceil\sfrac{\dist_{H}}{2} \right\rceil$. On the other hand, if we consider the reflected infection $\tilde{r}_{t}(\lceil\sfrac{d_{H}}{4}\rceil,0)$, that starts with the right half-axis infected and travels to the left, we obtain a similar statement for $B_{1}$. More precisely, the outcome of $f_{1}$ is not determined by the graphical construction restricted to $(-\infty, \left\lceil\sfrac{\dist_{H}}{2} \right\rceil) \times [0,s]$ if, and only if, the reflected infection reaches $B_{1}$ or it reaches the line  $y=\left\lceil\sfrac{\dist_{H}}{2} \right\rceil$. Besides, the graphical construction is independent in disjoint subsets of the space-time.

Let $A$ be the event where $r_{t}(\lceil\sfrac{3d_{H}}{4}\rceil,0)$ either touches $B_{2}$ or the line $y=\left\lceil\sfrac{\dist_{H}}{2} \right\rceil$,  and denote by $\tilde{A}$ the respective event with the infection $\tilde{r}_{t}(\lceil\sfrac{d_{H}}{4}\rceil,0)$ and the box $B_{1}$. If we choose $\usebigconstant{c:horizontal_decoupling_1}$ and $\usebigconstant{c:horizontal_decoupling_2}$ large enough, we can use Theorems~\ref{teo:finite_velocity} and~\ref{teo:vertical_decoupling} to bound
\begin{equation}
\PP_{\rho}[A] \leq ce^{-c^{-1}\log^{\sfrac{5}{4}} d_{H}}.
\end{equation}
By symmetry, the same is true for $\PP_{\rho}[\tilde{A}]$. We now can bound
\begin{equation}
\begin{split}
\EE_{\rho}[f_{1}f_{2}] & \leq \EE_{\rho}[f_{1}f_{2}\charf{A^{c} \cap \tilde{A}^{c}}]+\PP_{\rho}[A]+\PP_{\rho}[\tilde{A}]\\
& \leq \EE_{\rho}[f_{1}]\EE_{\rho}[f_{2}] + 2(\PP_{\rho}[A]+\PP_{\rho}[\tilde{A}]) \\
& \leq \EE_{\rho}[f_{1}]\EE_{\rho}[f_{2}] + ce^{-c^{-1}\log^{\sfrac{5}{4}} d_{H}}
\end{split}.
\end{equation}
The proof is complete.
\end{proof}

\appendix

\section{Proof of Proposition~\ref{prop:concentration}}\label{apdx:concentration}
~
\par This section contains the proof of Proposition~\ref{prop:concentration}. We present only the proof of the first statement, since the second one is obtained in the same way.

\par Begin by observing that
\begin{displaymath}
\EE_{\rho}[e^{\lambda X_{1}}]=\frac{Z(e^{\lambda}R^{-1}(\rho))}{Z(R^{-1}(\rho))}.
\end{displaymath}
By independence, for $\lambda>0$ we have
\begin{align*}
\PP_{\rho}\left[\sum_{k=1}^{n}X_{k} \geq (\rho+\epsilon)n\right] & = \PP_{\rho}\left[\exp\left\{\lambda \sum_{k=1}^{n}X_{k}\right\} \geq e^{\lambda(\rho+\epsilon)n}\right] \\
& \leq \left[\EE_{\rho}[e^{\lambda X_{1}}]e^{-\lambda(\rho+\epsilon)}\right]^{n} \\
& \leq \left[\frac{Z(e^{\lambda}R^{-1}(\rho))}{Z(R^{-1}(\rho))}e^{-\lambda(\rho+\epsilon)}\right]^{n}.
\end{align*}

\par We now split the last term above and work with the function
\begin{equation*}
f(\lambda)=\frac{Z(e^{\lambda}R^{-1}(\rho))}{Z(R^{-1}(\rho))}e^{-\lambda(\rho+\frac{\epsilon}{2})}.
\end{equation*}
Observe that $f(0)=1$ and that
\begin{equation*}
f'(\lambda)=e^{-\lambda(\rho+\frac{\epsilon}{2})}\frac{Z(e^{\lambda}R^{-1}(\rho))}{Z(R^{-1}(\rho))}\left[R(e^{\lambda}R^{-1}(\rho))-\rho-\frac{\epsilon}{2}\right].
\end{equation*}

\par The function $R$ is increasing. Besides, $R'$ is continuous, hence $R$ is Lipschitz continuous on the interval $[0,e\rho_{+}]$ and hence, for $ \lambda \leq 1$,
\begin{align*}
R(e^{\lambda}R^{-1}(\rho))-\rho & = R(e^{\lambda}R^{-1}(\rho))-R(R^{-1}(\rho)) \\
& \leq \tilde{c}(\rho_{+})R^{-1}(\rho_{+})(e^{\lambda}-1) < \frac{\epsilon}{2},
\end{align*}
for all $\lambda < \lambda_{*}(\epsilon) := \min\left\{\log\left(1+\frac{\epsilon}{2\tilde{c}(\rho_{+})R^{-1}(\rho_{+})}\right), 1\right\}$. For such values of $\lambda$, $f'(\lambda)<0$ and hence $f(\lambda) \leq 1$. Now, we just need to choose $c(\rho_{+})$ such that $2c(\rho_{+})\epsilon < \lambda_{*}(\epsilon)$ for all $\epsilon \leq 1$. This implies that we can bound
\begin{equation}\label{eq:final_upper_bound}
\PP_{\rho}\left[\sum_{k=1}^{n}X_{k} \geq (\rho+\epsilon)n\right] \leq e^{-c(\rho_{+})\epsilon^{2}n},
\end{equation}
completing the proof.

\section{Proof of Claim~\ref{claim:coupling}}\label{apdx:claim}
~
\par In this section we conclude the proof of Claim~\ref{claim:coupling}, that states that $(\eta_{s})_{s \geq 0}$ defined in Subsection~\ref{subsec:coupling} is indeed a zero range process.

\par We already observed that this is the case when $\eta_{0}$ has finitely many particles. We now treat the case when the number of particles of $\eta_{0}$ is infinite.

\par To verify that $(\eta_{s})_{s \geq 0}$ is a zero range process in this case, if suffices to check that the semigroup $(T_{s})_{s \geq 0}$ associated with $(\eta_{s})_{s \geq 0}$ coincides with the semigroup $(S_{s})_{s \geq 0}$ of a zero range process with rate function $g$. Fix then a local bounded continuous function $f: \NN_{0}^{[-n,n]} \to \RR$. We need to verify that
\begin{equation}
T_{t}f(\eta_{0})=S_{t}f(\eta_{0}) \quad \mu_{\rho}\text{-almost surely}.
\end{equation}

\par Let $(\eta_{s}^{m})_{s \geq 0}$ be the process with initial configuration
\begin{equation}
\eta_{0}^{m}(x)=\begin{cases}
\eta_{0}(x), & \quad \text{if } |x| \leq n+m, \\
0, & \quad \text{otherwise}.
\end{cases}
\end{equation}

\par We have
\begin{equation}
\begin{split}
S_{t}f(\eta_{0})& =\lim_{m \to \infty} S_{t}f(\eta_{0}^{m}) \\
& = \lim_{m \to \infty} T_{t}f(\eta_{0}^{m})
\end{split}.
\end{equation}
To conclude, we need to verify that
\begin{equation}
\PP_{\rho}\left[\eta_{t}^{m}(x) \neq \eta_{t}(x), \text{ for some } x \in [-n,n]\right] \to 0,
\end{equation}
as $m \to \infty$.

\par In the event above, there exists a particle that is outside $[-m-n,n+m]$ at time zero and reaches $[-n,n]$ before time $t$. Since particles move as random walks up to time changes, Lemma~\ref{lemma:many_particles} still applies. This implies that, if $X \sim \poisson(\Gamma_{+}tA(\sqrt{m},t))$, then
\begin{equation}
\begin{split}
\PP_{\rho}\left[\eta_{t}^{m}(x) \neq \eta_{t}(x), \text{ for some } x \in [-n,n]\right] & \leq 2\sum_{n+1}^{n+m}\PP_{\rho}\left[\begin{array}{cl}
\eta_{s}(0) \geq A(\sqrt{m}, t), \\ \text{for some } s \in [0,t]
\end{array}\right] \\
& \qquad + 2\PP_{\rho}[X \geq m] \\
& \leq 2\useconstant{c:many_particles}m(t+1)e^{-\useconstant{c:many_particles}^{-1}\sqrt{m}}+2e^{-m},
\end{split}
\end{equation}
if $m$ is large enough. If we let $m \to \infty$, the probability above converges to zero, and this implies
\begin{equation}
T_{t}f(\eta_{0}) =\lim_{m \to \infty} T_{t}f(\eta_{0}^{m}).
\end{equation}
In  particular, we have $T_{t}f(\eta_{0})=S_{t}f(\eta_{0})$, $\mu_{\rho}$-almost surely. This applies to any local function $f$ and concludes the proof.

\section{Proof of Proposition~\ref{prop:martingale_large_deviation}}
~
\par In this section we study the martingale introduced in \eqref{eq:martingale} and prove Proposition~\ref{prop:martingale_large_deviation}.

\newconstant{c:tail_bound}

\par The first lemma we prove is a tail bound for the increments of this martingale.
\begin{lemma}\label{lemma:tail_bound}
There exists a positive constant $\useconstant{c:tail_bound}$ that depends only on the density $\rho>0$ such that, for all $t \geq 0$ and $u \geq 0$,
\begin{equation}
\PP_{\rho}[|M_{t+1}-M_{t}| \geq u] \leq \useconstant{c:tail_bound}e^{-\useconstant{c:tail_bound}^{-1}u^{\sfrac{1}{2}}}.
\end{equation}
\end{lemma}

\begin{proof}
By translation invariance, it is enough to consider $t=0$ and, by increasing if necessary the value of $\useconstant{c:tail_bound}$, we can also consider $u \geq 1$.

Consider the event
\begin{equation}
A=\left\{\sup_{s \in [0,1]}|r_{s}-r_{0}| \geq \frac{u}{2}\right\},
\end{equation}
and observe that Lemmas~\ref{lemma:infection_1} and~\ref{lemma:infection_3} imply
\begin{equation}
\PP_{\rho}[A] \leq ce^{-c^{-1}u^{\sfrac{1}{2}}}.
\end{equation}

We also introduce the event
\begin{equation}
B=\left\{ \eta_{s}(x) \geq \frac{u}{\Gamma_{+}}, \text{ for some } (x,s) \in [-u,u] \times [0,1]\right\}.
\end{equation}
Union bound and Lemma~\ref{lemma:many_particles} gives
\begin{equation}
\PP_{\rho}[B] \leq ce^{-c^{-1}u}.
\end{equation}

Finally, on $(A \cup B \cup \{r_{0} \leq \sfrac{u}{2}\})^{c}$, using that $g(k) \leq \Gamma_{+}k$, we obtain
\begin{equation}
\begin{split}
|M_{1}-M_{0}| & \leq |r_{1}-r_{0}|+\left|\int_{0}^{1}\frac{1}{2}g(\eta_{s}(r_{s})) \charf{\eta_{s}(r_{s}) \geq 2} \,\mathrm{d}s\right| \\
& \leq \frac{u}{2}+\frac{1}{2}\Gamma_{+}\frac{u}{\Gamma_{+}} =u,
\end{split}
\end{equation}
and the proof is complete.
\end{proof}

\par Using the tail bound obtained above we can prove the concentration estimates for the martingale $M_{t}$. This proof follows the lines from~\cite{lv}.

\begin{proof}[Proof of Proposition~\ref{prop:martingale_large_deviation}]
We investigate the martingale $M_{t}$ restricted to the integer times. Denote $\mathcal{F}_{n}=\sigma(M_{t}: t \leq n)$ and $X_{n}=M_{n}-M_{n-1}$. Given a positive integer $k$ we define
\begin{equation}
Y_{n}=X_{n}\charf{\{|X_{n}| \leq L_{k}^{\sfrac{1}{4}}\}}-\EE_{\rho}\left[\left.X_{n}\charf{\{|X_{n}| \leq L_{k}^{\sfrac{1}{4}}\}}\,\right|\, \mathcal{F}_{n-1}\right],
\end{equation}
and
\begin{equation}
Z_{n}=X_{n}\charf{\{|X_{n}| > L_{k}^{\sfrac{1}{4}}\}}-\EE_{\rho}\left[\left.X_{n}\charf{\{|X_{n}| > L_{k}^{\sfrac{1}{4}}\}}\,\right|\, \mathcal{F}_{n-1}\right].
\end{equation}

Observe that both $Y_{n}$ and $Z_{n}$ are martingale differences with respect to the filtration $(\mathcal{F}_{n})_{n \geq 0}$ and that $Y_{n}+Z_{n}=X_{n}$.

We easily obtain that
\begin{equation}\label{eq:concentration_1}
\begin{split}
\PP_{\rho}[|M_{L_{k}}| \geq \delta L_{k}] & = \PP_{\rho}\left[\left|\sum_{n=1}^{L_{k}}X_{n}\right| \geq \delta L_{k}\right]\\
& \leq \PP_{\rho}\left[\left|\sum_{n=1}^{L_{k}}Y_{n}\right| \geq \frac{\delta}{2} L_{k}\right]+\PP_{\rho}\left[\left|\sum_{n=1}^{L_{k}}Z_{n}\right| \geq \frac{\delta}{2} L_{k}\right].
\end{split}
\end{equation}

We now focus on the two probabilities on the right hand side of the estimate above.

Notice that $|Y_{n}| \leq 2L_{k}^{\sfrac{1}{4}}$. Hence, Azuma's inequality implies
\begin{equation}\label{eq:concentration_2}
\PP_{\rho}\left[\left|\sum_{n=1}^{L_{k}}Y_{n}\right| \geq \frac{\delta}{2} L_{k}\right] \leq 2e^{-\frac{\delta^{2}}{32}L_{k}^{\sfrac{1}{2}}}.
\end{equation}

As for $Z_{n}$, observe initially that $F_{n}(u):=\PP_{\rho}[|X_{n}| \geq u] \leq \useconstant{c:tail_bound}e^{-\useconstant{c:tail_bound}^{-1}u^{\sfrac{1}{2}}}$, according to Lemma~\ref{lemma:tail_bound}. We now bound
\begin{equation}
\begin{split}
\EE_{\rho}[Z_{n}^{2}]& \leq \EE_{\rho}\left[X_{n}^{2}\charf{\{|X_{n}| > L_{k}^{\sfrac{1}{4}}\}}\right] \\
& = -\int_{L_{k}^{\sfrac{1}{4}}}^{\infty}x^{2} dF_{n}(x) \\
& = -\lim_{M \to \infty} \int_{L_{k}^{\sfrac{1}{4}}}^{M}x^{2} dF_{n}(x) \\
& = - \lim_{M \to \infty} \left( M^{2}F_{n}(M)-L_{k}^{\sfrac{1}{2}}F_{n}(L_{k}^{\sfrac{1}{4}})-\int_{L_{k}^{\sfrac{1}{4}}}^{M} 2xF_{n}(x) dx \right) \\
& \leq  L_{k}^{\sfrac{1}{2}}\useconstant{c:tail_bound}e^{-\useconstant{c:tail_bound}^{-1}L_{k}^{\sfrac{1}{8}}}+\int_{L_{k}^{\sfrac{1}{4}}}^{\infty} 2\useconstant{c:tail_bound}xe^{-\useconstant{c:tail_bound}^{-1}x^{\sfrac{1}{2}}} dx \\
& \leq cL_{k}^{\sfrac{3}{8}}e^{-cL_{k}^{\sfrac{1}{8}}},
\end{split}
\end{equation}
by possibly changing constants.

This implies that
\begin{equation}\label{eq:concentration_3}
\PP_{\rho}\left[\left|\sum_{n=1}^{L_{k}}Z_{n}\right| \geq \frac{\delta}{2}L_{k}\right] \leq \frac{4}{\delta^{2}L_{k}^{2}}\EE\left[\left(\sum_{n=1}^{L_{k}}Z_{n}\right)^{2}\right] \leq ce^{-cL_{k}^{\sfrac{1}{8}}}.
\end{equation}

Combining Equations~\eqref{eq:concentration_1},~\eqref{eq:concentration_2} and~\eqref{eq:concentration_3} easily implies the proposition.
\end{proof}

\bibliographystyle{plain}
\bibliography{mybib}

\end{document}